\documentclass[a4paper,10pt]{amsart}
\usepackage[utf8]{inputenc}
\usepackage[T1]{fontenc}

\usepackage{amssymb, mathtools, tensor, stmaryrd}
\usepackage{graphicx}

\usepackage{pstricks}

\usepackage{tikz, tikz-cd, xcolor}
\tikzcdset{
	column sep/normal=3.4em,
	column sep/large=4.5em
}

\allowdisplaybreaks

\usepackage{faktor}
\usepackage{enumitem}

\setitemize{label=$\rhd$, itemsep=0.45em, leftmargin=20pt}
\setenumerate{label=(\roman*),itemsep=0.2em}

\usepackage[
	colorlinks=true,
	citecolor={blue!80!black},
	linkcolor={orange!70!purple},
	urlcolor={olive!80!black},
	linktoc=section]
{hyperref}

\usepackage{setspace}

\usepackage[indentafter]{titlesec}
\usepackage{titletoc}
\titleformat{\section}
  {\scshape\bfseries\centering}
  {\thesection}{1em}{}

\titleformat{\subsection}[runin]
  {\normalfont\normalsize\bfseries\centering}
  {\thesubsection}{5pt}{}[.]

\titleformat{\subsubsection}[runin]
  {\normalfont\normalsize\bfseries\centering}
  {\thesubsubsection}{5pt}{}[.]

\titlecontents{section}[1.5em]
    {}
    {\normalfont\small\contentslabel{1.5em}}
    {\normalfont\small}
    {\dotfill\small\contentspage}{}
\titlecontents{subsection}[5em]
    {}
    {\normalfont\small\contentslabel{2.5em}}
    {}
    {\small\contentspage}{}

\newtheorem{thm}{Theorem}[section]
\newtheorem{prop}[thm]{Proposition}
\newtheorem{lem}[thm]{Lemma}
\newtheorem{cor}[thm]{Corollary}
\theoremstyle{definition}
\newtheorem{rem}[thm]{Remark}

\newtheorem{dfn}[thm]{Definition}
\numberwithin{equation}{section}

\makeatletter
\NewCommandCopy\@@pmod\pmod
\DeclareRobustCommand{\pmod}{\@ifstar\@pmods\@@pmod}
\def\@pmods#1{\mkern4mu({\operator@font mod}\mkern 6mu#1)}
\makeatother

\newcommand{\matrica}[1]
	{\left(\begin{array}{cc} #1 \end{array} \right)}
	
\newcommand\tm{\times}
\newcommand\otm{\otimes}

\newcommand{\N}{\mathbb{N}}
\newcommand{\Z}{\mathbb{Z}}
\newcommand{\RR}{\mathbb{R}}
\newcommand{\CC}{\mathbb{C}}
\newcommand{\KK}{\mathbb{K}}

\newcommand{\ra}{\longrightarrow}		% strelica nadedsno (za displayed equations)
\newcommand{\hra}{\hookrightarrow}		% injektivna strelica nadesno
\newcommand{\hla}{\hookleftarrow}		% injektivna strelica nalevo
\newcommand{\sra}{\twoheadrightarrow}	% surjektivna strelica nalevo

\newcommand{\xra}[1]{\xhookrightarrow{#1}}
\newcommand{\xla}[1]{\xhookleftarrow{#1}}

\newcommand{\calC}{\mathcal{C}} 		% proizv. kategorija
\renewcommand{\O}{\mathbb{O}}		% objekat neke kategorije C

\newcommand{\1}{\mathbb{I}}		% jedinica kao objekat
\newcommand\mj{{\bf 1}} 		% jedinica kao automorfizam

\renewcommand{\b}{\beta}
\newcommand{\be}{{\beta\unit}}
\newcommand{\g}{\gamma}
\newcommand{\Cyl}{\mathrm{Cyl}}		% mapping cylinder

\newcommand{\unit}{\eta}			% jedinica
\newcommand{\counit}{\varepsilon}	% kojedinica

\newcommand{\swap}{\mathrm{sw}}		% swap map

\newcommand{\Cob}{3\mathsf{Cob}}			% kategorija 3Cob
\renewcommand{\Vec}{\mathsf{Vec}}			% kategorija kompleksnih vekt.prostora
\newcommand{\MCG}{\mathrm{MCG}}				% mapping class group
\newcommand{\bMCG}{\b\mathsf{MCG}}			% kategorija generisana sa beta (pairing) i MCG
\newcommand{\beMCG}{\b\unit\mathsf{MCG}}	% kategorija generisana sa beta (pairing) epsilon (unit) i MCG

\newcommand{\SL}{SL(2,\Z)}			% grupa matrica det=1
\newcommand{\GL}{GL(2,\Z)}			% grupa invertibilnih matrica, dakle det=\pm1
\newcommand{\Mat}{Mat}				% the set of matrices
\newcommand{\tr}{\mathrm{tr}}		% trace

\newcommand{\Aut}{\mathrm{Aut}}

\newcommand{\wt}{\widetilde}

\newcommand{\wh}{\widehat}

\title{Restricted (2+1)-TQFTs supported by thickened and solid tori}

\author[D. {\DJ}or{\dj}evi\'c]{Du\v{s}an {\DJ}or{\dj}evi\'c}
\address{University of Belgrade, Faculty of Physics, Studentski trg 12, 11000 Belgrade, Serbia}
\email{dusan.djordjevic@ff.bg.ac.rs}

\author[D. Kosanovi\'c]{Danica Kosanovi\'c}
\address{Universit\"at Bern, Mathematisches Institut (MAI), Alpeneggstrasse 22, 3012 Bern, Switzerland}
\curraddr{CNRS, IMJ-PRG, Université Paris Cité, Sorbonne Université, 75013 Paris, France}
\email{danica.kosanovic@cnrs.fr}

\author[J. Nikoli\'c]{Jovana Nikoli\'c}
\address{University of Belgrade, Faculty of Mathematics, Studentski trg 16, 11000 Belgrade, Serbia}
\email{jovana.nikolic@matf.bg.ac.rs}

\author[Z. Petri\'c]{Zoran Petri\'c}
\address{Serbian Academy of Sciences and Arts, Mathematical Institute, Knez Mihailova 36, 11001 Belgrade, Serbia}
\email{zpetric@mi.sanu.ac.rs}

\begin{document}
\begin{abstract}
	A faithful $(1+1)$ TQFT has recently been constructed, but the existence of a faithful $(2+1)$ TQFT remains an open question, which subsumes the hard problem of linearity of mapping class groups of surfaces. To circumvent the latter problem we construct a subcategory of the category of 3-cobordisms, containing disjoint unions of tori and simplest cobordisms between them. On this we define TQFTs that are able to distinguish pairs of torus bundles and lens spaces, previously shown not to be distinguishable by quantum invariants.
\end{abstract}

\maketitle

A $(2+1)$-dimensional topological quantum field theory (TQFT) is a symmetric monoidal functor from the category $\Cob$ to finite dimensional complex vector spaces $\Vec_{\CC}$. The objects of $\Cob$ are oriented compact surfaces without boundary, and its morphisms are oriented 3-dimensional cobordisms, considered up to orientation-preserving homeomorphisms preserving boundary components. \emph{In this paper we are concerned with the question of existence of a faithful $(2+1)$ TQFT}, that is, a functor that is injective both on objects and morphisms.

To put this question into context, in one dimension lower it has long been known that $(1+1)$ TQFTs are in bijective correspondence with commutative Frobenius algebras. However, only recently a faithful one has been constructed, in the joint work of one of the present authors with Gajovi\'c and Telebakovi\'c-Oni\'c~\cite{G20}.

In the case of $(2+1)$ TQFTs the classification has recently been given in terms of J-algebras by Juh\'asz~\cite{Juhasz}, but these objects are very difficult to construct and classify. Moreover, note that the question of faithfulness of a $(2+1)$ TQFT subsumes that of linearity of mapping class groups ($\MCG$) of all surfaces. Indeed, the mapping cylinder $\Cyl_f$ of a homeomorphism $f: \Sigma\to\Sigma$ is a cobordism to which a TQFT $F$ assigns a linear map $F(\Cyl_f): F(\Sigma)\to F(\Sigma)$. If $F$ is faithful, then this gives an injection $\MCG(\Sigma)\hra GL(n,\CC)$ for some $n\geq1$.

The most well-studied source of $(2+1)$ TQFTs are modular tensor categories, defined by Turaev~\cite{Turaev}, following his work with Reshetikhin~\cite{Resh-Turaev}. The main examples arise as semisimple quotients of representation categories of quantum groups. However, already at the level of $\MCG$ representations these TQFTs fail to be faithful, as one can show that the action of any Dehn twist has finite order.

In fact, Dong, Lin, and Ng~\cite{DLN} showed that already for the torus $T$ each quantum representation of $\SL$ has nontrivial kernels of finite index (and is a congruence subgroup, that is, contains $\ker (\SL\sra SL(2,\Z/N\Z))$ for some $N>1$).

Building on this, Funar~\cite{F13} has shown that there exist infinitely many pairs $Bun_{K_n}$ and $Bun_{L_n}$ of non-homeomorphic closed 3-manifolds whose Reshetikhin--Turaev invariants are the same for all modular tensor categories. These pairs are the torus bundles over $S^1$ whose monodromies are homeomorphisms of the torus given by certain matrices $K_n,L_n\in \MCG(T)\cong \SL$, which are conjugate in $SL(2,\Z/N\Z)$ for every $N>1$; see also Theorem~\ref{thm:Funar2} below.

Importantly, modular categories have recently been studied in the non-semisimple setting, with examples in which Dehn twists have infinite order~\cite{BCGP,DGGPR,Muller-Woike}, and could potentially give faithful representations. See also the end of the introduction.

Therefore, our initial faithfulness question is heavily influenced by the theory of mapping class groups representations, and seems very hard. Since our viewpoint is rather that of 3-manifold theory, we will study a modification. Following \cite{F23} a \emph{restricted} $(2+1)$ TQFT is a symmetric monoidal functor defined on a symmetric monoidal subcategory of $\Cob$, with values in $\Vec_{\CC}$. \emph{Our goal is to construct faithful restricted $(2+1)$ TQFTs} on large subcategories of $\Cob$.

Firstly, we restrict to objects that are \emph{disjoint unions of tori}, for which the mapping class group is already itself linear. We restrict the class of 3-manifolds underlying the arrows, so that they are all obtained by gluing several copies of some simple generators. Namely, let $N^k$ denote the disk $D^2$ minus the interiors of $k\geq 0$ disjoint disks in the interior of $D^2$. Some simple classes of cobordisms are supported by thickened tori $N^1\tm S^1$ -- which includes cylinders and (co)pairings -- which generate the subcategory that we call $\bMCG<\Cob$. In particular, $\bMCG$ contains all torus bundles (as cobordisms from the empty set to itself); see Corollary~\ref{cor:support-b}. If we additionally include solid tori $N^0\tm S^1$ -- which represent (co)units -- we obtain the category called $\beMCG<\Cob$, which includes all lens spaces; see Corollary~\ref{cor:support-be}.

Our main results are summarised as follows. Let $\KK$ be $\RR$ or $\CC$.

\begin{thm}[Corollary~\ref{cor:TQFT}]
There is a subcategory $\beMCG<\Cob$ that contains all torus bundles and lens spaces, and such that a symmetric monoidal functor $F:\beMCG\to\Vec_{\KK}$ is completely described by the data:
	\begin{align*}
		n_F& \in\Z_{\geq1},\quad
		\rho^F_a  \in GL(n_F,\KK),\quad
		\rho^F_b \in GL(n_F,\KK),\\
		\b^F& \in \Mat(2n_F \tm 1,\KK),\quad
		\g^F \in \Mat(1\tm 2n_F,\KK),\quad
		\unit^F \in\Mat(1\tm n_F,\KK),
	\end{align*}
	that satisfy:
	\begin{align*}
		&\rho^F_a\rho^F_b\rho^F_a=\rho^F_b\rho^F_a\rho^F_b,\quad
		(\rho^F_a\rho^F_b)^6={ E_{n_F}},\\
		&\b^F \tau=\b^F, \quad
		\tau \g^F=\g^F,\quad
		\rho^F_a \unit^F = \unit^F,\quad 
		{ (\rho^F_a\rho^F_b)^3\unit^F = \unit^F,}\\
		&(\b^F\otm E_{n_F}) (E_{n_F}\otm\g^F)=E_{n_F}=(E_{n_F}\otm\beta^F) (\g^F\otm E_{n_F}),\\
		&\b^F (\rho^F_a\otm E_{n_F}) = \b^F (E_{n_F}\otm \rho^F_b),
	\end{align*}
	for the identity matrix $E_{n_F}\in GL(n_F,\KK)$ and $\tau=\matrica{0_{n_F} & E_{n_F}\\ E_{n_F} & 0_{n_F}}\in GL(2n_F,\KK)$.
\end{thm}
This follows from our main result about these subcategories, Theorem~\ref{thm:freely}, which describes them explicitly in terms of generators and relations. The key ingredients for this are the classical results, Theorems~\ref{thm:torus-bundles} and~\ref{thm:lens-spaces}, classifying the mentioned two classes of closed 3-manifolds, torus bundles and lens spaces.

As a corollary we can easily construct TQFTs that distinguish between pairs of 3-manifolds not distinguished by quantum invariants.

\begin{thm}[Theorem~\ref{thm:F3}]
	The functor ${ F_2}:\beMCG\to\Vec_{\RR}$ defined by
\begin{align*}
	&n_{ F_2}=3,\quad
	\rho^{ F_2}_a=\frac{1}{2}\begin{pmatrix}1 & 4 & 2\\0 & 2 & 2\\ 0 & 0 & 4\end{pmatrix},\quad
	\rho^{ F_2}_b=\frac{1}{2}\begin{pmatrix}4 & 0 & 0\\-2 & 2 & 0\\2 & -4 & 1\end{pmatrix},\quad
	\unit^{ F_2}=\begin{pmatrix}4 & 1 & 0 \end{pmatrix}^T,\\
	&\b^{ F_2}=\begin{pmatrix}2 & 4 & 1 & 4 & -4 & -4 & 1 & -4 & 2\end{pmatrix},\quad
	\g^{ F_2}=\frac{1}{36}\begin{pmatrix}8 & 4 & 4 & 4 & -1 & -4 & 4 & -4 & 8\end{pmatrix}^T,
\end{align*}
	is a restricted TQFT that distinguishes all Funar's torus bundle pairs not distinguished by any Turaev--Viro invariant, some Funar's torus bundle pairs not distinguished by any Reshetikhin--Turaev invariant, and some homotopy equivalent lens spaces.
\end{thm}

A natural next class of cobordisms to study is generated by $N^0\tm S^1$ and $N^2\tm S^1$. This gives precisely the class of \emph{graph manifolds} introduced in the seminal work of Waldhausen \cite{W67a, W67b} (see also \cite{FM97} where this class is called $\mathcal{H}$, and \cite{N05}). Equivalently, an irreducible graph manifold has only Seifert pieces in its JSJ decomposition. This is a very prominent class, for which we plan to investigate the existence of a faithful restricted $(2+1)$ TQFT in further work.

We use a somewhat non-standard definition of the category of 3-dimensional cobordisms. Namely, in each genus we choose one ``standard'' connected surface, and for non-connected surfaces we enumerate the connected components. This enables strict monoidal structure in our category, equivalent to the standard $\Cob$. Restricting to objects that are sequences of tori, we obtain exactly the skeleton of the corresponding subcategory of the standard $\Cob$. See Section~\ref{subsec:conventions}.

In fact, other variations of the category $\Cob$ are frequently considered. For example, in Turaev's category of ``extended cobordisms'' \cite[Chapter~IV]{Turaev} (see also \cite[Section~1]{Walker}) an object is a surface decorated with a Lagrangian subgroup of its first homology, and an arrow is a 3-manifold decorated with the signature of a 4-manifold that it bounds. On this category one can define a  TQFT associated to any semisimple modular category $\calC$, which descends to a $(2+1)$ TQFT on $\Cob$ if $\calC$ is anomaly-free.

Furthermore, if one requires that a cobordism has each connected component intersect the outgoing boundary, or a further decoration (by framed links or cohomology classes), then one can construct a TQFT associated to any non-semisimple modular category $\calC$. This has been intensively studied in recent years and has given promising results, see for example~\cite{BCGP,DGGPR,Muller-Woike}. Significantly, a TQFT constructed in \cite{BCGP} distinguishes \emph{all} lens spaces up to orientation preserving diffeomorphisms.
%This is related to the approach given in \cite[Chapter~IV]{Turaev} (see also \cite[Section~1]{Walker}) where the objects of $\Cob$ are surfaces parametrized by standard ones. This condition is simplified by replacing parametrization of a surface $\Sigma$ by a choice of Lagrangian subspace in $H_1(\Sigma;\mathbb{R})$. Furthermore, in order to use a modular category to construct a TQFT, cobordisms should be decorated. The functoriality of TQFT's in this case could be missed due to gluing anomalies, and ``killing'' the anomalies gives genuine functor from ``extended cobordisms'' to the category of $K$-modules for a ground ring $K$. A faithful such a TQFT, with or without anomaly, would be also considered as a positive answer to our question of faithfulness.

\subsection*{Acknowledgements}

We are grateful to the anonymous referee for pointing out a few issues and providing useful comments and suggestions that have considerably improved the paper.
The authors were partially supported by the Science Fund of the Republic of Serbia, Grant No. 7749891, Graphical Languages - GWORDS.

\tableofcontents

%%%
\section{Preliminaries}

\subsection{Conventions}
\label{subsec:conventions}
Let us first list some general notation.
\begin{enumerate}
	\item\label{conv:disj}
		We parametrise the disjoint union of spaces $X$ and $Y$ by
		\[
			X\sqcup Y\coloneqq \big( X\tm\{0\} \big) \cup \big( Y\tm\{1\} \big),
		\]
		and denote the swap map by
		\[
			\swap:X\sqcup Y\to Y\sqcup X,\quad  \begin{cases}(x,0) \mapsto (x,1),\\ (y,1)\mapsto (y,0).\end{cases}
		\]
	\item\label{conv:disj-incl}
		In particular, $X\sqcup X\subset X\tm I$, where $I\coloneqq [0,1]$. Given homeomorphisms $f,g:X\to X$ we use the same notation both for the homeomorphism $f\sqcup g: X\sqcup X\to X\sqcup X$ and for the embedding
		\[
			f\sqcup g: X\sqcup X\hra X\tm I,\quad \begin{cases}(x,0) \mapsto (f(x),0),\\ (x,1)\mapsto (g(x),1).\end{cases}
		\]
	\item\label{conv:swap}
		We will also use the embedding
		\[
			(f\sqcup g)\circ \swap:X\sqcup X\hra X\tm I,\quad \begin{cases}(x,0)\mapsto (g(x),1),\\ (x,1)\mapsto (f(x),0).\end{cases}
		\]
	\item\label{conv:time-incl-proj}
		Given $f:X\to X$ and $t\in I$, let us denote
		\[
			(f,t):X\hra X\tm I,\quad x\mapsto (f(x),t).
		\]
	\item
		For $t\in I$ we write $\pi_{X,t}:X\tm\{t\}\to X$ for the first projection, which we abbreviate by $\pi_t$ when $X$ is clear from the context.
	\item
		We write $M_0\cong M_1$ if and only if the oriented manifolds $M_0$ and $M_1$ are orientation preserving homeomorphic.
	\item\label{conv:boundary}
		To orient the boundary $\partial M$ of an oriented manifold $M$ we use the ``outward normal first'' convention.
		In particular, for an oriented manifold $\Sigma$, we give $\Sigma\tm I$ the product orientation, so the orientation of $\Sigma\tm\{0\}$ is such that $(id,0):\Sigma\hra\partial(\Sigma\tm I)$ is an orientation reversing embedding, whereas $(id,1):\Sigma\hra\partial(\Sigma\tm I)$ is orientation preserving.
	\item\label{conv:gluing}
		Given a subset $\Sigma\subseteq\partial M_1$ and a map $\varphi:\Sigma\to\partial M_0$ we define
		\[
			M_0\cup_\varphi M_1\coloneqq (M_0\sqcup M_1)/\sim,
		\]
		where $\sim$ is given by $(\varphi(x),0)\sim (x,1)$ for $x\in \Sigma$. If both $M_i$ are oriented manifolds, $\Sigma$ is the union of some components of $\partial M_1$, and $\varphi$ is an orientation reversing embedding, then $M_0\cup_\varphi M_1$ is an oriented manifold. This comes with natural orientation preserving inclusions {\black $M_i\hra M_0\cup_\varphi M_1$}.
		\item
		Let
		$E\coloneqq\matrica{1 & 0\\ 0 & 1},\,
		J\coloneqq\matrica{ 0 & 1 \\ 1 & 0 },\,
		 D_a\coloneqq\matrica{ 1 & 1 \\ 0 & 1 },\,
		 D_b\coloneqq\matrica{ 1 & 0 \\ -1 & 1 }$.
%	    \item We denote $T\coloneqq S^1\tm S^1$.
\end{enumerate}

Let us now discuss the category $\Cob$, whose objects are closed surfaces $\Sigma$ and arrows are 3-dimensional cobordisms $M$.
\begin{enumerate}[resume]
	\item
		A closed surface means a finite sequence of chosen representatives of connected oriented compact 2-manifolds without boundary (one standard surface for each genus). Since we are interested in tori only, for the unit circle $S^1$ we choose $S^1\times S^1$ to be the standard torus $T$, and our surfaces will be finite (possibly empty) sequences whose all members are $T$.
	\item\label{conv:equiv-of-cob}
		A cobordism from $\Sigma^s$ to $\Sigma^t$ means an oriented compact 3-manifold $M$ together with a homeomorphism $f^s\sqcup f^t$ from $\Sigma^s\sqcup\Sigma^t$ to $\partial M$, written
	\[
		\Sigma^t\xra{f^t} M\xla{f^s} \Sigma^s
	\]
		such that $f^s$ is \emph{orientation reversing}, while $f^t$ is \emph{orientation preserving}. Such manifolds $M$ are considered up to orientation preserving homeomorphisms $\Phi$ that preserve the boundary. More precisely, if the diagram
		\[\begin{tikzcd}
			\Sigma^t\rar[hook]{f^t_1}\dar[equals] & M_1\dar{\Phi}[swap]{\cong} &\lar[hook'][swap]{ f^s_1} \Sigma^s\dar[equals]\\
			\Sigma^t\rar[hook]{f^t_2} & M_2 &\lar[hook'][swap]{ f^s_2} \Sigma^s
		\end{tikzcd}
		\]
		commutes, then the two rows define the same arrow in $\Cob$.
	\item\label{conv:otm}
		The symmetric monoidal structure of $\Cob$ is given by concatenation $\otm$ on objects. For two arrows $C_0,C_1$ whose underlying manifolds are $M_0,M_1$, the product $C_0\otm C_1$ has $M_0\sqcup M_1$ as the underlying manifold, while the inclusions of boundaries factor through the inclusions $M_i\hookrightarrow M_0\sqcup M_1$. Although disjoint union is not strictly associative, by the definition of the equality of cobordisms in \ref{conv:equiv-of-cob}, this $\otm$ is associative, and gives a strict monoidal structure of $\Cob$. The symmetry transposes two surfaces in the sequence, as discussed later in~\eqref{eq:tau}.
	\item\label{conv:composition}
		The composition of arrows in $\Cob$ is given by gluing as in \ref{conv:gluing} the target of the cobordism $C_1$ to the source of $C_0$ by
		\begin{align*}
			C_0\circ C_1
			&= \big( \Sigma^t\xra{f^t_0} M_0\xla{f^s_0} \Sigma \big)\circ \big( \Sigma\xra{f^t_1} M_1\xla{f^s_1} \Sigma^s \big)\\
			&\coloneqq \big( \Sigma^t\xra{f^t_0} M_0\cup_{f^s_0\circ (f^t_1)^{-1}} M_1\xla{f^s_1} \Sigma^s \big).
		\end{align*}
	\item\label{conv:beta}
			If $f:\Sigma\to\Sigma$ preserves orientation and $g:\Sigma\to\Sigma$ reverses it, then using conventions~\ref{conv:disj-incl} and~\ref{conv:boundary} the map $f\sqcup g: \Sigma\sqcup\Sigma\hra \partial(\Sigma\tm I)$ is orientation reversing. Thus, there is a cobordism $(\emptyset\hra \Sigma\tm I \hla \Sigma\sqcup\Sigma: f\sqcup g)$. Similarly, there is a cobordism $(g\sqcup f:\Sigma\sqcup\Sigma \hra \Sigma\tm I \hla \emptyset)$.
	\item\label{conv:Cyl}
		If $f,g:\Sigma\to\Sigma$ are orientation preserving homeomorphisms, then there is a cobordism $(g,1):\Sigma\hra \Sigma\tm I \hla \Sigma : (f,0)$, using notation~\ref{conv:time-incl-proj}. We also briefly write this as $g:\Sigma\hra \Sigma\tm I \hla \Sigma : f$.
	\item\label{conv:self-gluing}
		For a cobordism $(g:\Sigma\hra M \hla \Sigma: f)$ the \emph{self-gluing} is the closed manifold obtained as the quotient of $M$ by the relation $f(x)\sim g(x)$ for every $x\in\Sigma$. This operation is sometimes called ``mending'' in the literature.
%	\item
%		We read schematic drawings of cobordisms from top to bottom. For example, the gluing from \ref{conv:gluing} and self-gluing from \ref{conv:self-gluing} are schematically represented in Figure~\ref{fig:gluing}.
\end{enumerate}
%\begin{figure}[htbp!]
%    \centering
%    \includegraphics[width=0.5\linewidth]{fig/fig-gluing}
%    \caption{The gluing and the self-gluing.}
%    \label{fig:gluing}
%\end{figure}

%%%
\subsection{Useful facts about (self-)gluing}

\begin{lem}[g-Collar]
\label{lem:collar}
	 Fix an oriented 3-manifold $M$ and
	 a closed (possibly disconnected) surface $\Sigma\subseteq\partial M$ with induced orientation. Given
 an orientation reversing embedding $g:\Sigma\to \partial M$ onto $\Sigma$
and an orientation preserving homeomorphism $f:\Sigma\to\Sigma$,
	there is an orientation preserving homeomorphism
	\[
		h:M\cup_{g\circ \pi_1} (\Sigma\tm I)\ra M
	\]
	such that $h\circ(f,0,1)=g\circ f$
	and $h$ keeps the other boundary components of $M$ fixed.
\end{lem}

Here we use the notation~\ref{conv:time-incl-proj} for the first projection $\pi_1:\Sigma\tm\{1\}\to\Sigma$, and notation~\ref{conv:gluing} for the gluing, so that $\Sigma\tm\{0\}\tm\{1\}\subset M\cup_{g\circ \pi_{1}} \Sigma\tm I$. Moreover, we generalise~\ref{conv:time-incl-proj} by defining $(f,0,1):\Sigma\to \Sigma\tm\{0\}\tm\{1\}$ by $x\mapsto (f(x),0,1)$. %Note that $\Sigma$ is not necessarily connected.

\begin{proof}
	We first define an orientation preserving homeomorphism
	\[
		h':M\cup_{g\circ\pi_1} (\Sigma\tm I) \ra M\cup_{i\circ\pi_0} (\Sigma\tm I),
	\]
	where $i:\Sigma\subseteq\partial M$. For $x\in M$, let $h'(x,0)=(x,0)$. For $(y,t)\in\Sigma\tm I$, let
	\[
		h'(y,t,1)= (g(y),1-t,1).
	\]
	We check that this is well defined, i.e.\ two points are related by $\sim$ in the source if and only if their $h'$ images are related by $\sim$ in the target. For the direction $(\Rightarrow)$ we know that these two points must be of the form $(g(x),0)$ and $(x,1,1)$ and their $h'$ images are $(g(x),0)$ and $(g(x),0,1)$, which are indeed related by $\sim$ in the target. The $(\Leftarrow)$ direction is similar. Moreover, note that $h'$ is the identity on $\partial M$ and
	\[
		h'\circ (f,0,1)(y)= h'(f(y),0,1)= (g(f(y)),1,1).
	\]
	Now note that the target of $h'$ is obtained from $M$ by attaching an external collar along $\Sigma$.
	By the Uniqueness of Collars Theorem, there is an orientation preserving homeomorphism $h'':M\cup_{i\circ\pi_0} (\Sigma\tm I)\to M$ such that $h''(y,1,1)= y\in\Sigma\subset M$ and $h''(y)=y$ for $y\in\partial M\setminus\Sigma$. Define $h:=h''\circ h'$, so $h\circ (f,0,1)= g\circ f$ as desired.
\end{proof}

If  in the statement of the lemma we replace $\pi_1:\Sigma\tm\{1\}\to\Sigma$ by $\pi_0:\Sigma\tm\{0\}\to\Sigma$, and ``reversing'' for $g$ by ``preserving'', then we have the same conclusion but with $(f,0,1)$ replaced by $(f,1,1)$. We have the following two corollaries of these results.

\begin{cor}\label{cor:collar1}
	$\big( \Sigma^t\xra{h} M\xla {g_2}\Sigma \big) \circ \big( \Sigma\xra {g_1}\Sigma\tm I\xla{f}\Sigma \big)
	= \big( \Sigma^t\xra{h} M\xla {g_2\circ g_1^{-1}\circ f}\Sigma \big)$.
\end{cor}

\begin{cor}\label{cor:collar2}
	$\big( \Sigma\xra{f} \Sigma\tm I\xla {g_1}\Sigma \big)\circ \big( \Sigma\xra {g_2} M\xla{h}\Sigma^s \big)
	= \big( \Sigma\xra {g_2\circ g_1^{-1}\circ f} M\xla{h}\Sigma^s \big)$.
\end{cor}
%\begin{figure}[htbp!]
%    \centering
%    \includegraphics[width=0.35\linewidth]{fig/fig-BunA}
%    \caption{A gluing as a self-gluing.}
%    \label{fig:BunA}
%\end{figure}

For the next helpful fact we use the swap map $\swap$ as in~\ref{conv:swap}, the notion of self-gluing from~\ref{conv:self-gluing}, and we give $\Sigma\times I$ the product orientation as in~\ref{conv:boundary}.
\begin{lem}[Self-Gluing]
\label{lem:self-gluing}
	Given an oriented surface $\Sigma$ and orientation preserving homeomorphisms $f,g:\Sigma\to\Sigma$, the manifold $\Sigma\tm I\cup_{(f\sqcup g)\circ \swap} \Sigma\tm I$ can be obtained as the self-gluing of the cobordism
$(g: \Sigma\hra \Sigma\tm I\hla\Sigma: f)$.
\end{lem}
\begin{proof}
	The manifold $\Sigma\tm I\cup_{(f\sqcup g)\circ \swap} \Sigma\tm I$ is obtained by gluing along two copies of $\Sigma\sqcup \Sigma$. This can be done in two steps. First, perform the gluing of $\Sigma\tm\{1\}\tm\{1\}$ to $\Sigma\tm\{0\}\tm\{0\}$ using $f$; second, glue $\Sigma\tm\{0\}\tm\{1\}$ to $\Sigma\tm\{1\}\tm\{0\}$ using $g$. Using the definition of gluing the first step can be written as the cobordism composition:
\[
	\big( \Sigma\xra{g} \Sigma\tm I\xla{f} \Sigma \big) \circ \big( \Sigma\xra{id}\Sigma\tm I\xla{id}\Sigma \big),
\]
	and the second step is precisely the self-gluing of this: $(x,0,1)\sim(g(x),1,0)$.	
	It remains only to observe that by Corollary~\ref{cor:collar1} this displayed composite cobordism is equal to $(g : \Sigma\hra \Sigma\tm I\hla \Sigma : f)$.
\end{proof}

%%%
\subsection{The mapping class group}

In the torus $T\coloneqq S^1\tm S^1$ consider the two oriented curves $a=S^1\tm\{pt\}$ and $b=\{pt\}\tm S^1$.
The right-handed Dehn twists $\rho_a$ and $\rho_b$ are homeomorphisms of $T$ drawn in Figure~\ref{fig:Dehn-twists}.
We recall the following fact, and refer to \cite{MCG} for details.
\begin{prop}\label{prop:MCG}
	The oriented mapping class group of the torus $T$ (i.e.\ the group of its orientation preserving homeomorphisms modulo isotopy) is isomorphic to the group $\SL$ of integral $2\tm2$ matrices of determinant one. The matrix
	\[
		\matrica{ p & r \\ q & s }\in\SL
	\]
	corresponds to a homeomorphism (unique up to isotopy) that satisfies $a\mapsto pa+qb$ and $b\mapsto ra+sb$ in homology.
\end{prop}
\begin{rem}
	If a matrix $A$ corresponds to a homeomorphism $\varphi_A$, then the composition rule is $\varphi_B\circ\varphi_A=\varphi_{BA}$, as usual.
\end{rem}

In particular, the right-handed Dehn twists on the curves $a$ and $b$ from Figure~\ref{fig:Dehn-twists} correspond to the matrices
	\begin{equation}\label{eq:Da,Db}
		D_a\coloneqq\matrica{ 1 & 1 \\ 0 & 1 } \quad {\rm and}\quad
		D_b\coloneqq\matrica{ 1 & 0 \\ -1 & 1 }.
	\end{equation}
\begin{figure}[htbp!]
    \centering
\begin{tikzpicture}[baseline=0cm]
	\draw[semithick] (0,0) ellipse (1.8cm and 1cm);
	\draw[semithick]
			(-0.8,0.05) to[distance=0.5cm,out=-30,in=-150] (0.8,0.05)
			(-0.6,-0.04) to[distance=0.5cm,out=30,in=150] (0.6,-0.04);
	\draw[red, thick]
			(0,-0.15) to[distance=0.25cm,out=-160,in=160] (0,-1)
			(-0.3,-0.3) node{\footnotesize$a$};
	\draw[red, thick, densely dotted]
			(0,-0.15) to[distance=0.25cm,out=-10,in=10] (0,-1);
	\draw[blue, thick]
			(-0.17,-1) to[distance=0.4cm,out=90,in=-90]
			(-1.3,0) to[distance=0.8cm,out=90,in=90]
			(1.3,0) to[distance=0.6cm,out=-90,in=-140]
			(0.15,-0.13)
			(0.7,-0.6) node{\footnotesize$b{+}a$};
	\draw[blue, thick, densely dotted]
			(-0.15,-1) to[distance=0.4cm,out=20,in=-70] (0.15,-0.13);
\end{tikzpicture}
\;$\xleftarrow{\rho_a}$
\begin{tikzpicture}[baseline=0cm]
	\draw[semithick] (0,0) ellipse (1.8cm and 1cm);
	\draw[semithick]
			(-0.8,0.05) to[distance=0.5cm,out=-30,in=-150] (0.8,0.05)
			(-0.6,-0.04) to[distance=0.5cm,out=30,in=150] (0.6,-0.04);
	\draw[red,thick]
			(0,-0.15) to[distance=0.25cm,out=-160,in=160] (0,-1)
			(-0.3,-0.3) node{\footnotesize$a$};
	\draw[red,thick,densely dotted]
			(0,-0.15) to[distance=0.25cm,out=-10,in=10] (0,-1);
	\draw[blue,thick]
			(0,0) ellipse (1.3cm and 0.6cm)
			(1.1,0) node{\footnotesize$b$};
\end{tikzpicture}\;$\xrightarrow{\rho_b}$
\begin{tikzpicture}[baseline=0cm]
	\draw[semithick] (0,0) ellipse (1.8cm and 1cm);
	\draw[semithick]
			(-0.8,0.05) to[distance=0.5cm,out=-30,in=-150] (0.8,0.05)
			(-0.6,-0.04) to[distance=0.5cm,out=30,in=150] (0.6,-0.04);
	\draw[red,thick]
			(0,-0.15) to[distance=0.5cm,out=-150,in=-90]
			(-1.1,0) to[distance=0.8cm,out=90,in=90]
			(1.5,0) to[distance=0.7cm,out=-90,in=150] (0,-1)
			(0.55,-0.76) node{\footnotesize$a{-}b$};
	\draw[red,thick,densely dotted]
			(0,-0.15) to[distance=0.25cm,out=-10,in=10] (0,-1);
	\draw[blue,thick]
			(0,0) ellipse (1.3cm and 0.6cm)
			(1.1,0) node{\footnotesize$b$};
\end{tikzpicture}
    \caption{The right-handed Dehn twists on $a$ and $b$.}
    \label{fig:Dehn-twists}
\end{figure}
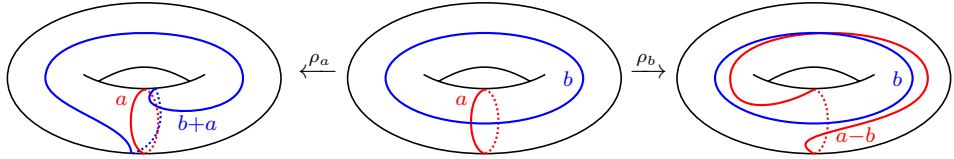

We will also make use of the homeomorphisms
	\begin{equation}\label{eq:E,J}
		E\coloneqq\matrica{1 & 0\\ 0 & 1},\quad J\coloneqq\matrica{ 0 & 1 \\ 1 & 0 },
	\end{equation}
where $J$ is orientation reversing. We have $J^2=E$ and for $P\neq Q$ in $\{D_a,D_b\}$:
\begin{equation}\label{eq:J}
	JP^{\pm}J=Q^{\mp}.
\end{equation}
\begin{rem}\label{row-column}
	If $A,\wh{A}\in\SL$ have the same first column, then $\wh{A}=A (D_a)^m$ for some $m\in \Z$. If $A,\wh{A}\in\SL$ have the same first row, then $\wh{A}=(D_b)^m A$ for some $m\in \Z$. The proofs follow from computing $A^{-1}\wh{A}$ and $\wh{A}A^{-1}$ respectively.
\end{rem}

The following can be extracted from~\cite[Chapter~9]{Karpenkov} for example. See also \cite{Fosse} for an alternative formula.
\begin{prop}\label{prop:SL-presentation}
	There is a presentation
\[
		\SL=\langle D_a,D_b \mid D_aD_bD_a=D_bD_aD_b,\; (D_aD_b)^6=1\rangle.
\]
	For $A=\matrica{ p & r \\ q & s }\in\SL$ let $[m_1,\dots,m_k]$ be a regular negative continued fraction expansion for $p/q$, namely:
 \[
 	\frac{p}{q}=m_1-\frac{1}{m_2-\frac{1}{\dots-\frac{1}{m_k}}} \,,
 \]
 with $m_1\in\Z$ and $m_i\in\Z_{>0}$ for $i=2,\dots,k$. Then the matrix
	\begin{equation}\label{eq:A-hat}
		\wh{A}=\mathrm{sgn}(p) D_a^{m_1-1}D_b^{-1}D_a^{m_2-2}\cdots D_b^{-1} D_a^{m_k-2}D_b^{-1}.
	\end{equation}
has the same first column as $A$, so $\wh{A}^{-1}A=D_a^m$ for some $m\in\Z$. In particular, $A=\wh{A}(\wh{A}^{-1}A)$ is an expression for $A$ in terms of the generators $D_a$ and $D_b$.
\end{prop}
% mathematica CODE for cont.fraction of x with n steps and sign b:
% continuedFraction[x_, n_, b_: 1] := Sign[b] Reap[NestWhile[b/# - Sow@Floor[b/#] &, Abs[b]/x, # != 0 &, 1, n]][[2, 1]];

\begin{rem}
	If $p/q$ has an expansion $[m_1,\dots,m_k]$ that yields the matrix $\wh{A}$ as in~\eqref{eq:A-hat}, then $[m_1,\dots,m_k+1,1]$ is also an expansion, and~\eqref{eq:A-hat} gives $\wh{A}D_a^{-1}$ instead.
\end{rem}

%%%
\subsection{Torus bundles}
\label{subsec:torus-bundles}

We will need some facts about torus bundles over the circle, which are not readily found in the literature.

Given $A\in\SL$ we can form the 3-manifold $Bun_A$ that is a bundle over $S^1$ with fibres homeomorphic to $T$, and with monodromy $A$. It is defined by
\begin{equation}\label{eq:BunA}
	Bun_A\coloneqq \faktor{T\tm I}{(Ax,0)\sim (x,1)}.
\end{equation}
The product orientation on $T\tm I$ induces an orientation on $Bun_A$.
\begin{lem}
\label{lem:BunA-as-gluing}
	We have $Bun_A\cong T\tm I\cup_{(A\sqcup E)\circ \swap} T\tm I$.
\end{lem}
\begin{proof}
	By the Self-Gluing Lemma~\ref{lem:self-gluing} the right hand side is equal to the self-gluing of the cobordism $(E : T\hra T\tm I\hla T : A)$. By definition, this is the quotient of $T\tm I$ by $(Ax,0)\sim(Ex,1)$ for every $x\in T$, so the claim follows using $Ex=x$.
\end{proof}

\begin{lem}
\label{lem:any-BunA}
	A torus bundle over $S^1$ has the form $Bun_A$ for some $A\in\SL$.
\end{lem}
\begin{proof}
	Any fibre bundle over the interval is trivial, i.e.\ homeomorphic to $T\tm I$. This means that for a fibre bundle $X$ over $S^1$, the restrictions $X_1,X_2$ of $X$ to two semicircles are of the shape $T\tm I$. We obtain $X$ by gluing to $T\tm\{i\}\subset T\tm I\cong X_1$ the end $T\tm\{1-i\}\subset T\tm I\cong X_2$, for $i=0,1$ via two homeomorphisms $B,C\in\SL$. In other words, 	
	\[
		 X\cong T\tm I\cup_{(B\sqcup C)\circ \swap} T\tm I.
	\]
	By Lemma~\ref{lem:self-gluing} this is equivalent to the self-gluing of $(C: T\hra T\tm I\hla T :B)$.
	Thus, $X$ is the quotient of $T\tm I$ by $(Bx,0)\sim (Cx,1)$ for all $x\in T$, or equivalently by $(BC^{-1}y,0)\sim (y,1)$ for all $y\in T$. This gives $Bun_{BC^{-1}}$ by definition~\eqref{eq:BunA}.
\end{proof}

% ALTERNATIVNI DOKAZ
%   Equivalently, we can first apply $C\tm id_I$ to the whole second copy, and then glue one end via $BC^{-1}$ and the other via $E$. In other words, there is a homeomorphism
%	\[
%	\Phi: X\cong T\tm I\cup_{(B\sqcup C)\circ \swap} T\tm I\to T\tm I\cup_{(BC^{-1}\sqcup E)\circ \swap} T\tm I\eqqcolon Bun_{BC^{-1}},
%	\]
%	given for $(x,t)\in T\tm I$ by $\Phi((x,t),0)=((x,t),0)$ and $\Phi((x,t),1)=((Cx,t),1)$. This is well defined since
%\begin{align*}
%	\text{ in the source }\mapsto & \text{ in the target } \\
%	((x,0),1)\sim((Cx,1),0)\mapsto &((Cx,0),1)\sim((Cx,1),0)\iff ((y,0),1)\sim((y,1),0),\\
%	((x,1),1)\sim((Bx,0),0)\mapsto &((Cx,1),1)\sim((Bx,0),0)\iff ((y,1),1)\sim((BC^{-1}y,0),0),
%\end{align*}
%and is clearly an orientation preserving homeomorphism.
%

We can prove the following directly; for an alternative proof see Proposition~\ref{prop:classify-new-proof}.
\begin{lem}
\label{lem:conjugation}
	If $A,C\in \SL$ we have orientation preserving homeomorphisms
\[
	Bun_{CAC^{-1}} \cong Bun_A \cong Bun_{C(JA^{-1}J)C^{-1}}.
\]
\end{lem}
\begin{proof}
	The first homeomorphism is given by
	\[\begin{tikzcd}[column sep=large, row sep=small]
		Bun_A\dar[equals]\rar{\Phi_1} & Bun_{CAC^{-1}}\dar[equals]\\
		\faktor{T\tm I}{(Ax,0)\sim (x,1)}\rar{C\tm id_I} &
		\faktor{T\tm I}{(CAC^{-1}y,0)\sim (y,1)}
	\end{tikzcd}
	\]
	In other words, $\Phi_1(x,t)=(Cx,t)$. This is well defined since $\Phi_1(x,1)=(Cx,1)$ agrees with $\Phi_1(Ax,0)=(CAx,0)$, by putting $x=C^{-1}y$.
	
	Now it suffices to show that $JA^{-1}J$ gives the same bundle as $A$. For this we have a homeomorphism
	\[\begin{tikzcd}[column sep=large]
		Bun_A\dar[equals]\rar{\Phi_2} & Bun_{JA^{-1}J}\dar[equals]\\
		\faktor{T\tm I}{(Ax,0)\sim (x,1)}\rar{J\tm (1-id_I)} &
		\faktor{T\tm I}{(JA^{-1}Jy,0)\sim (y,1)}
	\end{tikzcd}
	\]
	given by $\Phi_2(x,t)=(Jx,1-t)$. This is well defined since $\Phi_2(x,1)=(Jx,0)$ agrees with $\Phi_2(Ax,0)=(JAx,1)$, by putting $x=A^{-1}Jy$.
\end{proof}

%\begin{lem}[For later use]
%	$A^{\pm1}$ is conjugate in $\SL$ to $B$ or $JBJ$ if and only if $A^{\pm1}$ is conjugate to $B$ in $GL(2,\Z)$.
%\end{lem}
%\begin{proof}
%	If $B=CA^{\pm1}C^{-1}$ for $X\in GL(2,\Z)$, then $JBJ=(JX)A^{\pm1}(JX)^{-1}$ and $JX\in\SL$.
%\end{proof}

In fact, these are precisely all the relations among the torus bundles, thanks to the following theorem. This can be found in \cite[Lemma~6.2]{N81} without proof, and in \cite[Proposition~1.2, Remark~1.3]{F13} with a different proof.
\begin{thm}
\label{thm:torus-bundles}
	The map $A\mapsto Bun_A$ is a bijection
	\[\begin{tikzcd}
		\faktor{\SL}{\substack{A\sim CAC^{-1}\\A\sim JA^{-1}J}}\rar & \faktor{\{\text{oriented torus bundles over }S^1\}}{\cong}.
	\end{tikzcd}
	\]
\end{thm}
\begin{proof}
	By Lemma~\ref{lem:conjugation} the displayed map is well defined, and by Lemma~\ref{lem:any-BunA} it is surjective. To show that it is injective we follow closely the proof from Hatcher's Lectures~\cite{Hatcher} of a similar, but non-oriented statement. Let
	\[
	\Phi:Bun_A\to Bun_B
	\]
	be an orientation preserving homeomorphism, and consider $\Sigma\coloneqq \Phi(T\tm\{0\})$. This is an incompressible surface $\Sigma\subset Bun_B$ because $\Phi$ is a homeomorphism, and $T\tm\{0\}$ is incompressible in $Bun_A$. Then \cite[Lemma~2.7]{Hatcher} says that $\Sigma$ is either
	\begin{enumerate}[label=\arabic*)]
		\item isotopic to a torus fibre in $Bun_B$, or
		\item one can assume that $B$ is conjugate to one of the matrices $\pm\matrica{ 1 & n \\ 0 & 1 }$.
	\end{enumerate}
Consider the case 1). Without loss of generality we can assume that $\Sigma$ is the fibre in $Bun_B$ given by $T\tm\{0\}$. In other words, $\Phi$ sends $T\tm\{0\}$ to itself, so we have a well-defined self-homeomorphism $\wt{\Phi}$ of $T\tm I$ that lifts $\Phi$ under the quotient maps, i.e.\ makes the following square commute:
\[\begin{tikzcd}
	T\tm I\rar{\wt{\Phi}}\dar & T\tm I\dar\\
	Bun_A\rar{\Phi} & Bun_B.
\end{tikzcd}
\]
Then we have either $\wt{\Phi}|_0:T\tm\{0\}\to T\tm\{0\}$ or $\wt{\Phi}|_0:T\tm\{0\}\to T\tm\{1\}$. Since $\Phi$ is orientation preserving, in the first case $\wt{\Phi}|_0$ has to be an orientation preserving homeomorphism of the torus, so represented by an element $C\in\SL$. Similarly, in the second case the orientation of the interval is reversed, so the orientation of $T$ has to be as well, so $\wt{\Phi}|_0$ is represented by $CJ$ for some $C\in\SL$. Moreover, $\wt{\Phi}$ is a pseudo-isotopy from $\wt{\Phi}|_0$ to $\wt{\Phi}|_1$, so they are represented by the same matrix. Therefore, we have one of the following two cases:
\[\begin{tikzcd}
	T\tm \{1\}\rar{C}\dar{A} & T\tm \{1\}\dar{B}
														& T\tm \{1\}\rar{CJ}\dar{A} & T\tm \{0\}\\
	T\tm \{0\}\rar{C} & T\tm \{0\}
														& T\tm \{0\}\rar{CJ} & T\tm \{1\}\uar{B}
\end{tikzcd}
\]
That is, either $B=CAC^{-1}$ or $B=CJA^{-1}JC^{-1}$. Since these are both contained in the relations in the source, the desired map is injective in this case.

For the case 2), we observe that the proof of \cite[Theorem~2.6]{Hatcher} shows that $Bun_B$ is a Seifert fibered manifold, and that there is a choice of a homeomorphism $\Phi$ for which $\Sigma$ is a torus fibre in $Bun_B$. Then the case 1) finishes the proof.
\end{proof}

\begin{rem}
	Note that there is a homeomorphism
	\[
		Bun_A\coloneqq \faktor{T\tm I}{(Ax,0)\sim (x,1)}
		\ra \faktor{T\tm I}{(A^{-1}y,0)\sim (y,1)} \eqqcolon Bun_{A^{-1}}
	\]
	given by $(x,t)\mapsto (x,1-t)$ (put $y=Ax$). But this is \emph{orientation reversing}, which is not allowed under our equivalence relation on cobordisms; see~\ref{conv:equiv-of-cob}. This explains why Hatcher's non-oriented classification~\cite{Hatcher} reads: $Bun_A\cong Bun_B$ if and only if $A^{\pm1}$ is conjugate to $B$ in $GL(2,\Z)$.
\end{rem}

%%%%%%%%
\section{Some cobordisms and relations between them}
\label{sec:thick-tori}
%%%

Our subcategory $\bMCG$ of $\Cob$, mentioned in the introduction and defined in Section~\ref{sec:two-cats} below, will contain all cobordisms that are supported by the 3-manifold $T\tm I$, and its multiple copies. In particular, following the conventions \ref{conv:beta} and \ref{conv:Cyl} for writing cobordisms, we consider:
\begin{align}
	\text{the \emph{identity} }\quad  &\mj\coloneqq \big( E : T\hra T\tm I\hla T : E \big), \label{eq:Id}\\
	\text{the \emph{pairing} }\quad   &\b\coloneqq \big( \emptyset\hra T\tm I\hla  T\sqcup T : E\sqcup J \big),\label{eq:beta}\\
	\text{the \emph{copairing} }\quad &\g\coloneqq \big( J\sqcup E : T\sqcup T \hra T\tm I\hla \emptyset \big),\label{eq:gamma}\\
\intertext{and for any $A\in\SL$ its mapping cylinder:}
	\text{ the \emph{cylinder} }\quad &\Cyl_A\coloneqq \big( E :T\hra T\tm I\hla T: A \big),\label{eq:Cyl}\\
\intertext{as well as the following cobordism supported by two copies of the thick torus:}
	\text{the \emph{symmetry} }\quad &\tau\coloneqq \big( \swap : T\sqcup T \hra (T\sqcup T)\tm I\hla T\sqcup T : id \big),\label{eq:tau}
\end{align}
where we use the swap map $\swap$ from \ref{conv:disj}. To be explicit in $\tau$ we have
\[
(\swap,1):\begin{cases}
	(x,0)\mapsto & ((x,1),1)\\
	(x,1)\mapsto & ((x,0),1)
\end{cases}
\quad \text{and}\quad
\begin{rcases}
	((x,0),0) &\mapsfrom (x,0)\\
	((x,1),0) &\mapsfrom (x,1)
\end{rcases}: (id,0).
\]
\begin{rem}
	Note that $\Cyl_{E}=\mj$. Moreover, note that $\b$ and $\g$ use the map $J$ from~\eqref{eq:E,J} on exactly one of the tori in $T\sqcup T$, as the two boundary components of $T\tm I$ are oriented oppositely; the choice of $J$ is without loss of generality.
\end{rem}

\begin{lem}[Snake Identities]
\label{lem:snake}
	We have equalities of cobordisms from $T$ to $T$:
	\[
		(\b\otm\mj)\circ(\mj\otm\g)\; = \;\mj\; = (\mj\otm\b)\circ(\g\otm\mj).
	\]
\end{lem}
\begin{proof}
	We prove the first equality, and the proof of the second is similar.
	Throughout the proof we write $T\sqcup T\sqcup T\coloneqq (T\tm\{0\})\cup (T\tm\{1\})\cup (T\tm\{2\})$ and $M\coloneqq T\tm I$. By $x$ we denote a point in $T$. From definitions we have
	\[
	\b\otm\mj= \big( f^t:T\hra M\tm\{0\}\cup M\tm\{1\}\hla T\sqcup T\sqcup T:f^s \big)
	\]
	with
\[
f^t: \begin{cases}
	x\mapsto & (x,1,1)
\end{cases}
\qquad \text{and}\qquad
\begin{rcases}
	(x,0,0) &\mapsfrom (x,0)\\
	(Jx,1,0) &\mapsfrom  (x,1)\\
	(x,0,1) &\mapsfrom  (x,2)
\end{rcases}:f^s,
\]
	and
	\[
		\mj\otm\g=\big( g^t:T\sqcup T\sqcup T\hra M\tm\{2\}\cup M\tm\{3\}\hla T :g^s \big)
	\]
	with
\[
g^t:\begin{cases}
	(x,0)\mapsto & (x,1,2)\\
    (x,1)\mapsto & (Jx,0,3)\\
	(x,2)\mapsto & (x,1,3)
\end{cases}
\qquad \text{and}\qquad
\begin{rcases}
	(x,0,2)\mapsfrom x
\end{rcases}:g^s.
\]
	Composing these two cobordisms gives
	\[
	(\b\otm\mj)\circ(\mj\otm\g)= \big( f^t:T\hra X\hla T:g^s \big)
	\]
	with $X=\big( M\tm\{0\}\cup M\tm\{1\}\big) \cup_{f^s\circ (g^t)^{-1}} (M\tm\{2\}\cup M\tm\{3\})$
	and
\[
f^s\circ(g^t)^{-1}:\begin{cases}
    (x,1,2)\mapsto & (x,0,0) \\
	(x,0,3)\mapsto & (x,1,0) \\
	(x,1,3)\mapsto & (x,0,1)
\end{cases}
\]
	We define a map $h:X\to M=T\tm I$ by
\[h:\begin{cases}
	(x,t,2) \mapsto & (x,\tfrac{t}{4}) \\
	(x,t,0) \mapsto & (x,\tfrac{1+t}{4}) \\
	(x,t,3) \mapsto & (x,\tfrac{2+t}{4}) \\
	(x,t,1) \mapsto & (x,\tfrac{3+t}{4})
\end{cases}
\]
	This is a well-defined homeomorphism, such that $h\circ g^s(x)=h(x,0,2)=(x,0)$ and $h\circ f^t(x)=h(x,1,1)=(x,1)$. Therefore, $(\b\otm\mj)\circ(\mj\otm\g)$ is equivalent to $\mj$.
\end{proof}

Throughout the rest of this section we fix arbitrary matrices $A,B\in\SL$.

%%%%%%%%%%%
\subsection{Cylinders as cobordisms}
We record some properties of cobordisms~\eqref{eq:Cyl}.
\begin{lem}
\label{lem:cylinders}
	The following cobordisms are all equivalent
	\begin{align*}
			A^{-1}B :\;  T\hra  &T\tm I\hla  T \;: E,\\
		  	B :\;  T\hra  &T\tm I\hla  T \;: A,\\
			E :\;  T\hra  &T\tm I\hla  T \;: B^{-1}A.
	\end{align*}
	Moreover, the gluing of cobordisms is given by $\Cyl_{BA}=\Cyl_B\circ\Cyl_A$.
\end{lem}
\begin{proof}
This is an easy consequence of Corollaries~\ref{cor:collar1} and \ref{cor:collar2}.
\end{proof}
%\begin{proof}
%	The following commutative diagram gives desired homeomorphisms:
%	\[\begin{tikzcd}
%		 T\rar[hook]{A^{-1} B}\dar[equals] & T\tm I\dar{A\tm id_I} & T\lar[hook'][swap]{E}\dar[equals]\\
%		 T\rar[hook]{B} & T\tm I & T\lar[hook'][swap]{A}\\
%		 T\rar[hook]{E}\uar[equals] & T\tm I\uar{B\tm id_I} & T\lar[hook'][swap]{B^{-1} A}\uar[equals]
%	\end{tikzcd}
%	\]
%	Using this we can write $\Cyl_B=(B^{-1}: T\hra  T\tm I\hla  T:E)$, and easily compose:
%	\[
%		\Cyl_B\circ\Cyl_A =
%		( T\xra{B^{-1}}  T\tm I\xla{E}  T)\circ (T\xra{E}  T\tm I\xla{A}  T)=(T\xra{B^{-1}} T\tm I\xla{A}  T).
%	\]
%	This is equal to $(E : T\hra  T\tm I\hla  T:BA)=\Cyl_{BA}$ as desired.
%\end{proof}

By definition~\ref{conv:equiv-of-cob}, cobordisms $\Cyl_A$ and $\Cyl_B$ are equivalent if and only if there is an orientation preserving homeomorphism $\Phi$ of the underlying $3$-manifold, which commutes with boundary inclusions. That is, if and only if there is a commutative diagram
\[
\begin{tikzcd}
		 T\rar[hook]{E}\dar[equals] &  T\tm I\dar{\Phi} &  T\lar[hook'][swap]{A}\dar[equals]\\
		 T\rar[hook]{E} & T\tm I &  T\lar[hook'][swap]{B}
	\end{tikzcd}
\]
This means that $\Phi|_{T\tm\{0\}}\circ A=B$ and $\Phi|_{T\tm\{1\}}=E$, so $\Phi$ is precisely a pseudo-isotopy from $\Phi|_{T\tm\{0\}}=BA^{-1}$ to $E$. We then use the following classical result.

\begin{thm}[{\cite[Theorem~1.9]{M65}}, {\cite[Theorem~6.3]{E66}}]\label{thm:htpic-iff-istpic}
	Two surface homeomorphisms are homotopic if and only if they are pseudo-isotopic if and only if they are isotopic.
\end{thm}
Therefore, $\Cyl_A=\Cyl_B$ if and only if $A=B$. In fact, we have the following.

\begin{cor}
\label{cor:Cyl}
	For the group $\Aut(T)$ of automorphisms of the object $T$ in $\Cob$, the map $\Cyl:\SL\to\Aut(T)$, $A\mapsto\Cyl_A$, is an isomorphism of groups.
\end{cor}
% Our proof of this is different than the one given here https://mathoverflow.net/questions/155380/mapping-class-group-vs-automorphism-group-in-cobordism-category
\begin{proof}
	This is a homomorphism by Lemma~\ref{lem:cylinders}, and is injective by Theorem~\ref{thm:htpic-iff-istpic}.
	
	Let us show that any invertible arrow of $T$ to itself is a cylinder. Fix a cobordism $\varphi\in\Aut(T)$ in $\Cob$. Then there exists $\psi\in\Aut(T)$ with $\varphi\circ\psi=\mj$. In particular, the manifold $T\tm I$ supporting $\mj$ contains an embedded torus which divides it into $\varphi$ and $\psi$. This torus cannot be compressible (since on both sides it bounds a manifold that has one more boundary component), so it must be $\partial$-parallel, by the proof of \cite[Lemma~2.7]{Hatcher}, already mentioned in the proof of Theorem~\ref{thm:torus-bundles}. Thus, $\varphi$ or $\psi$ is supported by $T\tm I$, so the other one is as well.
\end{proof}

From Corollaries~\ref{cor:collar1},~\ref{cor:collar2} and~\ref{cor:Cyl}, one easily concludes the following.

\begin{cor}\label{beta-gamma-cyl}
	If $\beta\circ(\Cyl_A\otm id)=\beta\circ(\Cyl_B\otm id)$ or $(\Cyl_A\otm id)\circ\gamma=(\Cyl_B\otm id)\circ\gamma$, then $A=B$.
\end{cor}

%%%%%%%%%%%
\subsection{Interactions of cylinders and (co)pairings}
We will often use the following equalities; recall from \ref{conv:otm} the operation $\otm$ on cobordisms, given by the disjoint union of the underlying manifolds.
\begin{lem}\label{lem:beta-CylA-CylB}\label{lem:J-sigma}
	We have
	\begin{align*}
		\b\circ (\Cyl_A\otm\Cyl_B) &= \big( \emptyset\hra T\tm I\hla  T\sqcup T :A\sqcup JB \big)\\
		&=\big( \emptyset\hra T\tm I\hla  T\sqcup T : (B\sqcup JA)\circ \swap \big)\\
		(\Cyl_A\otm\Cyl_B)\circ\g &= \big( JA\sqcup B: T\sqcup T \hra T\tm I\hla\emptyset \big)\\
		&=\big( (JB\sqcup A)\circ \swap : T\sqcup T \hra T\tm I\hla \emptyset \big).
	\end{align*}
\end{lem}
\begin{proof}
Using definitions we find
	\begin{align*}
		&\b\circ (\Cyl_A\otm\Cyl_B)\\
		&\coloneqq \big( \emptyset\hra T\tm I\xla{E\sqcup J}  T\sqcup T \big)\circ \big((T\xra{E} T\tm I\xla{A} T)\otm (T\xra{E} T\tm I\xla{B} T)\big)\\
		&=\big( \emptyset\hra T\tm I\xla{E\sqcup J}  T\sqcup T \big)\circ \big(T\sqcup T\xra{E\sqcup E } (T\sqcup T)\tm I \xla{A\sqcup B} T\sqcup T\big).
	\end{align*}
	By Corollary~\ref{cor:collar1} for $M=T\tm I$, $\Sigma=T\sqcup T$, $g_2=E\sqcup J$, and $g_1=E\sqcup E$, $f=A\sqcup B$, this cobordism is equivalent to $\emptyset\hra  T\tm I \hla T\sqcup T :g_2\circ g_1^{-1}\circ f= A\sqcup JB$.
	
	To prove the second equality for $\beta$, we define
\[
	\begin{tikzcd}[column sep=large]
		 \emptyset\rar[hook]\dar[equals] &  T\tm I\dar{\Phi} &  T\sqcup T\lar[hook'][swap]{A\sqcup JB}\dar[equals]\\
		 \emptyset\rar[hook] & T\tm I &  T\sqcup T\lar[hook'][swap]{(B\sqcup JA)\circ \swap}
	\end{tikzcd}
\]	
	by letting $\Phi(x,t)=(Jx,1-t)$. This is an orientation preserving homeomorphism (as it reverses the orientations both of $T$ and $I$). The diagram commutes since $\Phi(A\sqcup JB)(x,0)=\Phi(Ax,0)=(JAx,1)$ and $\Phi(A\sqcup JB)(x,1)=\Phi(JBx,1)=(Bx,0)$. By notation~\ref{conv:swap} this agrees with $(B\sqcup JA)\circ \swap$, as desired.
	
	The last two equalities are proven analogously.
\end{proof}

\begin{lem}
\label{lem:cob-supp-by-TxI}
	Any cobordism in $\Cob$ supported by $T\tm I$ is a composite of $\b,\g,\Cyl_A$.
\end{lem}
\begin{proof}
	By Lemma~\ref{lem:cylinders} any cobordism $(B : T\hra  T\tm I\hla  T:A)$ is equivalent to $\Cyl_{B^{-1}A}$. Moreover, for $A\in\SL$ and $A'\in\GL\setminus\SL$ by Lemma~\ref{lem:beta-CylA-CylB} any $(\emptyset \hra  T\tm I\hla T\sqcup T:A\sqcup A')$ is equal to the composite $\b\circ (\Cyl_A\otm\Cyl_{JA'})$, and any $(A'\sqcup A :T\sqcup T\hra  T\tm I\hla \emptyset)$ can be obtained as $(\Cyl_{JA'}\otm\Cyl_A)\circ\g$. This covers all possibilities for cobordisms supported by $T\tm I$.
\end{proof}

Next, we use Lemma~\ref{lem:J-sigma} to show compatibility of (co)pairing with the symmetry.
\begin{lem}
\label{lem:tau}
	We have $\b\circ\tau=\b$ and $\tau\circ\g=\g$.
\end{lem}
\begin{proof}
	We prove the first equality, and the proof of the second is similar. By the definition~\eqref{eq:beta} of $\beta$ and~\eqref{eq:tau} of $\tau$, we have
	\[
		\b\circ\tau= \big( \emptyset\hra T\tm I\xla{E\sqcup J} T\sqcup T\big)\circ \big(T\sqcup T \xra{(\swap,1)} (T\sqcup T)\tm I\xla{(id,0)} T\sqcup T \big).
	\]
	By Corollary~\ref{cor:collar1} and using $\swap^{-1}=\swap$ this cobordism is equivalent to
	\[
		\emptyset\hra  T\tm I \hla T\sqcup T : (E\sqcup J)\circ \swap^{-1}\circ id= (E\sqcup J)\circ \swap,
	\]
	which is in turn equal to $\b$ by applying Lemma~\ref{lem:J-sigma} to $A=B=E$.
\end{proof}

\begin{lem}[The $\b$-jump Identity]
\label{lem:beta-jump}
	We have
	\[
		\b\circ (\Cyl_{JB^{-1}JA}\otm\mj)
		=\b\circ (\Cyl_A\otm\Cyl_B)
		=\b\circ (\mj\otm\Cyl_{JA^{-1}JB}).
	\]
\end{lem}
\begin{proof}
	For the first equality the desired homeomorphism is given by
	\begin{equation}\label{eq:beta-homeo}
	\begin{tikzcd}[column sep=large]
		 \emptyset\rar[hook]\dar[equals]
		 & T\tm I\dar[swap]{JB^{-1}J\tm id_I}
		 & T\sqcup T\lar[hook']{A\sqcup JB}\dar[equals]\\
		 \emptyset\rar[hook]
		 & T\tm I
		 & T\sqcup T\lar[hook']{JB^{-1}JA\sqcup J}
	\end{tikzcd}
	\end{equation}
	Indeed, by Lemma~\ref{lem:beta-CylA-CylB} the top row describes the cobordism $\b\circ (\Cyl_A\otm\Cyl_B)$,
	whereas the bottom row describes $\b\circ (\Cyl_{JB^{-1}JA}\otm\mj)$.
		
	For the other identity we can use Lemma~\ref{lem:J-sigma} to first write
	\[
		\b\circ (\Cyl_A\otm\Cyl_B) =(\emptyset\hra T\tm I\hla  T\sqcup T :B\sqcup JA),
	\]
	and then use the homeomorphism $JA^{-1}J\tm id_I$.
\end{proof}

%%%%%%%%%%%
\subsection{Torus bundles as cobordisms}

Recall from Section~\ref{subsec:torus-bundles} the notation $Bun_A$ for the torus bundle over $S^1$ with monodromy $A\in\SL$. This is an oriented closed 3-manifold $Bun_A$ which we view as a cobordism from the empty set to itself.
\begin{prop}
\label{prop:Bun}
	The cobordism $\b\circ (\Cyl_A\otm\mj)\circ\g$ is equivalent to $Bun_A$.
\end{prop}
\begin{proof}
	By Lemma~\ref{lem:beta-CylA-CylB} we have $\b\circ (\Cyl_A\otm\mj)=(\emptyset\hra T\tm I\hla  T\sqcup T :A\sqcup J)$ and also $\g=((J\sqcup E)\circ \swap:T\sqcup T \hra T\tm I\hla \emptyset)$, so
	\[
		\b\circ (\Cyl_A\otm\mj)\circ\g=(\emptyset\hra T\tm I\cup_{g} T\tm I \hla \emptyset),
	\]
	for the map $g$ given by
 	\[
 	(A\sqcup J)\circ((J\sqcup E)\circ \swap)^{-1}=(A\sqcup J)\circ \swap\circ(J\sqcup E)=(A\sqcup J)\circ(E\sqcup J)\circ \swap=(A\sqcup E)\circ \swap.
 	\]
 By Lemma~\ref{lem:BunA-as-gluing} this is exactly $Bun_A$.
\end{proof}

From the classification of torus bundles in Theorem~\ref{thm:torus-bundles} we deduce the following.
\begin{cor}\label{cor:Bun-conj}
	For $A,B\in \SL$ we have
\[
	\b\circ(\Cyl_A\otm \mj)\circ\g = \b\circ(\Cyl_B\otm \mj)\circ\g
\]
	if and only if $A$ is conjugate in $\SL$ to $B$ or to $JB^{-1}J$.
\end{cor}

In contrast to the ``$\beta$-jump identity'' from Lemma~\ref{lem:beta-jump}, if both $\b$ and $\g$ are used, $\Cyl_A$ and $\mj$ can swap places, and no conjugation with $J$ occurs. More precisely:
\begin{lem}[The $\b$-$\g$-jump Identity]
\label{lem:beta-gamma-jump}
	For every $A\in \SL$ we have
\begin{equation}\label{eq:jump-Bun}
	\b\circ(\Cyl_A\otm \mj)\circ\g = \b\circ(\mj\otm \Cyl_A)\circ\g.
\end{equation}
\end{lem}

\begin{proof}
	 The cobordism $\b\circ(\mj\otm \Cyl_A)\circ\g$ agrees with $\b\circ(\Cyl_{JA^{-1}J}\otm\mj)\circ\g$ by Lemma~\ref{lem:beta-jump}, and by the last corollary this is equivalent to $\b\circ(\Cyl_A\otm \mj)\circ\g$.
	 \end{proof}

%\begin{rem}
%	One direction in Corollary~\ref{cor:Bun-conj} can be derived from Lemma~\ref{lem:beta-gamma-jump}.
%\end{rem}

%%%%%%%%
\subsection{Solid tori as cobordisms}
\label{sec:solid-tori}

In $\Cob$ we also consider the cobordism
\begin{equation}
\label{eq:unit}
	\text{the \emph{unit}}\quad \unit \coloneqq(\iota:T \hra D^2\tm S^1\hla \emptyset),
\end{equation}
where $\iota:T=S^1\tm S^1\hra D^2\tm S^1$ is the product of the inclusion of the boundary $S^1\hra D^2$ on the first factor (the curve $a$) and the identity on the second factor.

\begin{lem}
\label{lem:Da-e}
	We have $\Cyl_{D_a}\circ \unit = \unit$.
\end{lem}
\begin{proof}
	By definition~\eqref{eq:unit} and Corollary~\ref{cor:collar2} the cobordism $\Cyl_{D_a}\circ \unit$ is equal to
	\[
	(T\xra{E} T\tm I\xla{D_a} T)\circ (T \xra{\iota} D^2\tm S^1\hla \emptyset)= (T\xra{\iota\circ D_a^{-1}} D^2\tm S^1 \hla \emptyset).
	\]
	This is equivalent to $\unit$, i.e.\ there is a homeomorphism $\Phi:D^2\tm S^1\to D^2\tm S^1$ such that $\Phi\circ \iota\circ D_a^{-1}=\iota$. This follows from the fact that the Dehn twist around the meridian can be extended to the solid torus (see~\cite{McC} for a construction of $\Phi$).
\end{proof}

{ The following is proved in a similar manner by relying on $\Phi\colon D^2\tm S^1\to D^2\tm S^1$ defined by $\Phi(x,y,e^{it})=(x,-y,e^{-it})$ for $(x,y)\in D^2$ and $e^{it}\in S^1$.

\begin{lem}
\label{lem:-E-e}
	We have $\Cyl_{-E}\circ \unit = \unit$.
\end{lem}}

We also consider the cobordism
\begin{equation}\label{eq:counit}
	\text{the \emph{counit}}\quad \counit\coloneqq (\emptyset \hra D^2\tm S^1\hla T :\iota\circ J).
\end{equation}

\begin{lem}
\label{lem:e-eta}
	We have $\counit=\b\circ(\mj\otm\unit)$.
\end{lem}
\begin{proof}
	By definition, $\b\circ(\mj\otm\unit)$ is given by
	\begin{align*}
		 (\emptyset\hra T\tm I\xla{E\sqcup J} T\sqcup T)\circ \big( (T\xra{E} T\tm I\xla{E} T)\otm (T \xra{\iota} D^2\tm S^1\hla \emptyset) \big).
	\end{align*}
	In this composition we perform two gluings: first glue $T\tm\{0\}$ in $\mj$ to $T\tm\{0\}$ in $\b$ via $E\circ E^{-1}=E$ (and this does nothing), then glue the boundary of $D^2\times S^1$ to $T\tm\{1\}$ using $J\circ\iota^{-1}$.
	Equivalently, this glues $T\times I$ to $D^2\times S^1$ by sending $T\tm\{1\}$ to $\partial(D^2\times S^1)$ using the projection $\pi_{1}:T\tm\{1\}\to T$, followed by $J$ and $\iota$, so:
	\[
	\b\circ(\mj\otm\unit) =(\emptyset\hra D^2\tm S^1\cup_{\iota\circ J\circ \pi_{1}} T\tm I \hla T :E)
	\]
	By the g-Collar Lemma~\ref{lem:collar} this is equivalent to $\counit=(\emptyset \hra D^2\tm S^1\hla T :\iota\circ J)$.
\end{proof}

\subsection{Lens spaces as cobordisms}
\begin{dfn}
	{ For $p,q$ co-prime and $p\geq 0$,} define the lens space
	\[
		L(p,q)\coloneqq H_1\cup_{\varphi_{p,q}} H_2
	\] 
	where $H_i=D^2\tm S^1$ and the homeomorphism $\varphi_{p,q}:\partial H_2\to\partial H_1$ is in homology given by
	\[
		\varphi_{p,q}\big[ \partial D^2\tm\{pt\} \big]= q \big[ \partial D^2\tm \{pt\} \big] + p \big[ \{pt\}\tm S^1 \big].
	\]
	In particular, $L(p,q)$ is the $p/q$-surgery on the unknot in $S^3$ and has fundamental group $\Z/p\Z$.
\end{dfn}
Note that we follow the notation of \cite{PS}, and that $L(p,q)$ is denoted $L_{q/p}$ in \cite{Hatcher}.

\begin{lem}\label{lem:cob-is-lens-space}
	For $A=\matrica{p & r \\ q & s}\in\SL${, where $p\geq 0$,} the cobordism $\counit\circ\Cyl_A\circ \unit$
	is homeomorphic to the lens space $L(p,q)$.
\end{lem}
\begin{proof}
	By Corollary~\ref{cor:collar1} we have
	\[
		 \counit\circ \Cyl_A = (\emptyset \hra D^2\tm S^1\xla{\iota\circ J} T)\circ (T\xra{E} T\tm I\xla{A} T)  = (\emptyset \hra D^2\tm S^1\xla{\iota\circ J\circ A} T).
	\]
	Therefore,
	\begin{align*}
        \counit\circ\Cyl_A\circ \unit &=  (\emptyset \hra D^2\tm S^1\xla{\iota\circ J\circ A} T) \circ (T \xra{\iota} D^2\tm S^1\hla \emptyset)\\
        &=(\emptyset\hra D^2\tm S^2\cup_{\iota\circ J\circ A\circ{\iota}^{-1}} D^2\tm S^1\hla \emptyset).
    \end{align*}
    This is of the shape $H_1\cup_\varphi H_2$ for the map $\varphi=J\circ A=\matrica{q & s \\ p & r}$.
    This sends $a$ to $q a + p b$, so we indeed have $\counit\circ\Cyl_A\circ \unit=L(p,q)$.
\end{proof}

{  Until the end of this section we assume that $A=\matrica{ p & r \\ q & s }$ and $A'=\matrica{ p' & r' \\ q' & s' }$ are from $\SL$, with $p,p'\geq 0$.}

\begin{dfn}
	{ We say that $A$ and $A'$ are} \emph{1-step lens-inseparable} when either $A'=D_b^{-n}AD_a^k$ for some $n,k\in\Z$ or $A'=JA^{-1}J$. We say that two matrices are \emph{lens-inseparable} when there is a finite sequence of matrices connecting them, so that any consecutive pair is 1-step lens-inseparable.
\end{dfn}

\begin{lem}\label{lem:insep-cond}
	{ The matrices $A$ and $A'$} are lens inseparable if and only if $p=p'$ and either  $q\equiv q'\pmod*{p}$ or $qq'\equiv 1\pmod*{p}$.
\end{lem}
\begin{proof}	
	If $A'=D_b^{-n}AD_a^k$ then comparing entries gives $p'=p$, $q'=q+np$, $r'=r+kp$ and $s'=s+nr+k(q+np)$. In particular, $q'\equiv q\pmod*{p}$. If $A'=JA^{-1}J$ then $p'=p$, $q'=-r$, $r'=-q$, and $s'=s$. In particular, $ps'-q'(-q)=p's'-q'r'=1$, being the determinant of $A'\in\SL$, so $qq'\equiv1\pmod*{p}$.
	
	In the other direction, if $p'=p$ and $q'=q+np$ then $A'=D_b^{-n}AD_a^{(s+nr)r'-s'r}$ is 1-step lens-inseparable from $A$. In the other case, $p=p'$ and $qq'=1+pm$ for some $m\in\Z$, so
\[
	A=\matrica{ p & r \\ q & s },\;
	B=\matrica{ p & -q' \\ q & -m },\;
	JB^{-1}J=\matrica{ p & -q \\ q' & -m },\;
	A'=\matrica{ p' & r' \\ q' & s' }
\]
	is a sequence of 1-step lens-inseparable matrices that connects $A$ and $A'$.
\end{proof}

\begin{thm}\label{thm:lens-spaces}
	The cobordisms $\counit\circ\Cyl_A\circ \unit$ and $\counit\circ\Cyl_{A'}\circ \unit$ are equivalent if and only if the matrices $A$ and $A'$ are lens-inseparable.
\end{thm}
\begin{proof}
	By a classical result of Reidemeister, lens spaces $L(p,q)$ and $L(p',q')$ are homeomorphic if and only if $p=p'$ and either $q\equiv q'$ or $qq'\equiv 1$ modulo $p$. These conditions are by Lemma~\ref{lem:insep-cond} equivalent to lens-inseparability of the matrices $A$ and $A'$, for which $\counit\circ\Cyl_A\circ \unit=L(p,q)$ and $\counit\circ\Cyl_{A'}\circ \unit=L(p',q')$ by Lemma~\ref{lem:cob-is-lens-space}.
\end{proof}

We will later also need the following corollary of this theorem.

\begin{lem}\label{cyl-epsilon}\label{eta-cyl}
	If $\Cyl_A\circ\unit= \Cyl_{A'}\circ\unit$, then $A'=AD_a^m$ for some $m\in\Z$.
	If $\counit\circ\Cyl_A= \counit\circ\Cyl_{A'}$, then $A'=D_b^mA$ for some $m\in\Z$.
\end{lem}

\begin{proof}
	We prove the first claim, and the proof of the second is similar.
	The equality $\Cyl_A\circ\unit= \Cyl_{A'}\circ\unit$ implies $\counit\circ\Cyl_A\circ\unit= \counit\circ\Cyl_{A'}\circ\unit$, as well as 
	\[ 
	\counit\circ \Cyl_{D_a}\circ\Cyl_A\circ\unit= \counit\circ \Cyl_{D_a}\circ\Cyl_{A'}\circ\unit.
	\]  
	Then Theorem~\ref{thm:lens-spaces} says that $A$ and $A'$ are lens-inseparable, as well as $D_aA$ and $D_aA'$. { The former implies $p=p'$, and the latter $p+q=p+q'$, hence $q=q'$.} Now Remark~\ref{row-column} finishes the proof.
\end{proof}

%%%%%%%%
\section{Two subcategories of \texorpdfstring{$\Cob$}{Cob}}
\label{sec:two-cats}

We consider the following two categories.

\begin{dfn}\label{dfn:bMCG}
\begin{enumerate}
\item
	Let $\bMCG$ be the subcategory of $\Cob$ whose objects are sequences of the object $T$, and whose arrows are generated by $\mj,\b,\g,\Cyl_A$, using the monoidal structure $\otm$, composition $\circ$, and symmetry~$\tau$.

\item
	Let $\beMCG$ be the subcategory of $\Cob$ whose objects are sequences of the object $T$, and whose arrows are generated by $\mj,\b,\g,\Cyl_A$ and additionally $\unit$, using the monoidal structure $\otm$, composition $\circ$, and symmetry $\tau$.
\end{enumerate}
\end{dfn}

Let us now fix an arbitrary symmetric strict monoidal category $\calC$.
In Section~\ref{subsec:Tb-objects} we introduce the notion of a $T^\b$-object in $\calC$, and in Section~\ref{subsec:Tbe-objects} of a $T^\be$-object in $\calC$, and we explore their properties. In Section~\ref{subsec:universality} we show that the torus $T\in\bMCG$ is a universal example of a $T^\b$-object, and $T\in\beMCG$ is a universal $T^\be$-object.

%%%%%%%%%%%
\subsection{\texorpdfstring{$T^\b$}{Tb}-objects}
\label{subsec:Tb-objects}

\begin{dfn}\label{dfn:Tb-object}
	Fix a symmetric, strict monoidal category $\calC$, with the unit object $\1$, the symmetry $\tau$, and the identity arrows $\mj$. A $T^\b$-\emph{object} in $\calC$ is an object $\O$ together with the following data.
\begin{itemize}
\item
	There is a group action of $\SL$ on $\O$, that is, a group homomorphism
	\[
		\rho^\O:\SL\ra \Aut(\O).
	\]
\item
	There is a non-degenerate, symmetric pairing $\b^\O:\O\otm \O\to \1$, with the corresponding copairing $\g^\O:\1\to \O\otm \O$. Explicitly, this means
	\begin{align}
		\b^\O\circ\tau=\b^\O&,\label{eq:tau-beta}\\
		\tau\circ\g^\O=\g^\O&,\label{eq:tau-gamma}\\
		(\b^\O\otm\mj)\circ(\mj\otm\g^\O)=\mj&=(\mj\otm\beta^\O)\circ(\g^\O\otm\mj).\label{eq:snake}
	\end{align}
\item
	Moreover, for $D_a,D_b\in\SL$ from~\eqref{eq:Da,Db} the following equality holds
\begin{equation}\label{eq:Tb-object}
	\b^\O\circ(\rho^\O_{D_a}\otm \mj) = \b^\O\circ(\mj\otm \rho^\O_{D_b}).
\end{equation}
\end{itemize}
\end{dfn}
\begin{rem}
One can show that the second equality in~\eqref{eq:snake} follows from the first, using symmetry and its naturality. We include both of them for convenience.
\end{rem}

\begin{lem}\label{lem:Tb-variation}
The following equality holds for every $T^\b$-object $\O$:
\begin{equation}
	(\rho^\O_{D_a}\otm \mj)\circ\g^\O= (\mj\otm \rho^\O_{D_b})\circ\g^\O.
	\label{eq:Tb-object-gamma}
\end{equation}
Moreover, we have the equalities obtained from~\eqref{eq:Tb-object} and~\eqref{eq:Tb-object-gamma} by exchanging $D_a$ and $D_b$, or by replacing $D_a$ with $D^{-1}_a$ and $D_b$ with $D^{-1}_b$.
\end{lem}
\begin{proof}
	The equality obtained from~\eqref{eq:Tb-object} by exchanging $D_a$ and $D_b$ is proven by
\begin{equation*}
\begin{split}
	\b^\O\circ(\rho^\O_{D_b}\otm \mj)
	= \b^\O\circ\tau\circ(\rho^\O_{D_b}\otm \mj)
	&= \b^\O\circ(\mj\otm \rho^\O_{D_b})\circ\tau\\
	& = \b^\O\circ(\rho^\O_{D_a}\otm \mj)\circ\tau \\
	&= \b^\O\circ\tau\circ(\mj\otm \rho^\O_{D_a})
	= \b^\O\circ(\mj\otm \rho^\O_{D_a}).
\end{split}
\end{equation*}
For the equality obtained by replacing $D_a$ with $D^{-1}_a$ and $D_b$ with $D^{-1}_b$, we just have to precompose~\eqref{eq:Tb-object} with $(\rho^\O_{D_a^{-1}} \otm \rho^\O_{D_b^{-1}} )$.

We prove~\eqref{eq:Tb-object-gamma} using the ``snake identities''~\eqref{eq:snake} and the key equality~\eqref{eq:Tb-object}. The proof is presented diagrammatically in Figure~\ref{fig:diag-proof-Tb-object-gamma} and algebraically by:
\begin{equation*}
\begin{split}
	(\rho^\O_{D_a}\otm \mj)\circ\g^\O
	& = ((\mj\circ\rho^\O_{D_a})\otm \mj)\circ\g^\O
\\
	& = \Big(\big( (\mj\otm\beta^\O)\circ(\g^\O\otm\mj)\circ \rho^\O_{D_a} \big)\otm \mj \Big)\circ\g^\O
\\
	& = \Big(\big( (\mj\otm\beta^\O)\circ(\mj\otm\mj\otm \rho^\O_{D_a})\circ(\g^\O\otm\mj) \big)\otm \mj \Big)\circ\g^\O
\\
	& = \Big( \mj\otm \big( \beta^\O\circ(\mj\otm\rho^\O_{D_a}) \big)\otm \mj \Big) \circ (\g^\O\otm\g^\O)
\\
	& = (\mj\otm(\beta^\O\circ(\rho^\O_{D_b}\otm\mj))\otm \mj) \circ (\g^\O\otm\g^\O)
\\
& = \Big( \mj\otm \big( (\beta^\O\otm \mj)\circ(\mj\otm\g^\O) \circ \rho^\O_{D_b} \big) \Big)\circ\g^\O
\\
& = (\mj\otm\rho^\O_{D_b})\circ\g^\O.
\end{split}
\end{equation*}
	For the last two claims in the case of~\eqref{eq:Tb-object-gamma} we proceed as before.
\end{proof}
\begin{figure}[htbp!]
    \centering
\begin{tikzpicture}[baseline=0.2cm,scale=0.8]
	\draw[red, thick] 		(-0.3,-0.4) -- (-0.3,0.35);
	\draw[thick] 			(-0.3,0.4) to[distance=0.55cm,out=90,in=90] (0.3,0.4);
	\draw[thick] 			(0.3,-0.4) -- (0.3,0.35);
\end{tikzpicture}\quad=\:
\begin{tikzpicture}[baseline=-0.2cm,scale=0.8]
	\draw[thick] 			(-0.3,-1.5) -- (-0.3,-0.45);
	\draw[red, thick] 		(-0.3,-0.4) -- (-0.3,0.35);
	\draw[thick] 			(-0.3,0.4) to[distance=0.55cm,out=90,in=90] (0.3,0.4);
	\draw[thick] 			(0.3,-1.5) -- (0.3,0.35);
\end{tikzpicture}\quad=\:
\begin{tikzpicture}[baseline=-0.2cm,scale=0.8]
	\draw[thick] 			(-1.5,-1.4) -- (-1.5,-0.85);
	\draw[thick] 			(-1.5,-0.8) to[distance=0.55cm,out=90,in=90] (-0.9,-0.8);
	\draw[thick] 			(-0.9,-0.85) to[distance=0.55cm,out=-90,in=-90] (-0.3,-0.85);
	\draw[thick] 			(-0.3,-0.8) -- (-0.3,-0.45);
	\draw[red, thick] 		(-0.3,-0.4) -- (-0.3,0.35);
	\draw[thick] 			(-0.3,0.4) to[distance=0.55cm,out=90,in=90] (0.3,0.4);
	\draw[thick] 			(0.3,-1.4) -- (0.3,0.35);
\end{tikzpicture}\quad=\quad
\begin{tikzpicture}[baseline=-0.2cm,scale=0.8]
	\draw[thick] 			(-1.5,-1.4) -- (-1.5,-0.05);
	\draw[thick] 			(-1.5,0) to[distance=0.55cm,out=90,in=90] (-0.9,0);
	\draw[thick] 			(-0.9,-0.8) -- (-0.9,-0.05);
	\draw[thick] 			(-0.9,-0.85) to[distance=0.55cm,out=-90,in=-90] (-0.3,-0.85);
	\draw[red, thick] 		(-0.3,-0.8) -- (-0.3,-0.05);
%	\draw[thick] 			(-0.3,0) -- (-0.3,0.35);
	\draw[thick] 			(-0.3,0) to[distance=0.55cm,out=90,in=90] (0.3,0);
	\draw[thick] 			(0.3,-1.4) -- (0.3,-0.05);
\end{tikzpicture}\quad=\quad
\begin{tikzpicture}[baseline=-0.2cm,scale=0.8]
	\draw[thick] 			(-1.5,-1.4) -- (-1.5,-0.05);
%	\draw[thick] 			(-1.5,0) -- (-1.5,0.35);
	\draw[thick] 			(-1.5,0) to[distance=0.55cm,out=90,in=90] (-0.9,0);
%	\draw[thick] 			(-0.9,0) -- (-0.9,0.35);
	\draw[blue, thick] 	(-0.9,-0.8) -- (-0.9,-0.05);
	\draw[thick] 			(-0.9,-0.85) to[distance=0.55cm,out=-90,in=-90] (-0.3,-0.85);
	\draw[thick] 			(-0.3,-0.8) -- (-0.3,-0.05);
	\draw[thick] 			(-0.3,0) to[distance=0.55cm,out=90,in=90] (0.3,0);
	\draw[thick] 			(0.3,-1.4) -- (0.3,-0.05);
\end{tikzpicture}\quad=\:\:
\begin{tikzpicture}[baseline=-0.2cm,scale=0.8]
	\draw[thick] 			(-1.5,-1.4) -- (-1.5,0.35);
	\draw[thick] 			(-1.5,0.4) to[distance=0.55cm,out=90,in=90] (-0.9,0.4);
	\draw[blue, thick] 	(-0.9,-0.35) -- (-0.9,0.35);
	\draw[thick] 			(-0.9,-0.8) -- (-0.9,-0.4);
	\draw[thick] 			(-0.9,-0.85) to[distance=0.55cm,out=-90,in=-90] (-0.3,-0.85);
	\draw[thick] 			(-0.3,-0.8) to[distance=0.55cm,out=90,in=90] (0.3,-0.8);
	\draw[thick] 			(0.3,-1.4) -- (0.3,-0.85);
\end{tikzpicture}\:=\quad
\begin{tikzpicture}[baseline=0.2cm,scale=0.8]
	\draw[thick] 			(-0.3,-0.4) -- (-0.3,0.35);
	\draw[thick] 			(-0.3,0.4) to[distance=0.55cm,out=90,in=90] (0.3,0.4);
	\draw[blue, thick] 		(0.3,-0.4) -- (0.3,0.35);
\end{tikzpicture}
    \caption{A diagrammatic proof of~\eqref{eq:Tb-object-gamma}, with red $\rho^\O_{D_a}$ and blue $\rho^\O_{D_b}$.}
    \label{fig:diag-proof-Tb-object-gamma}
\end{figure}
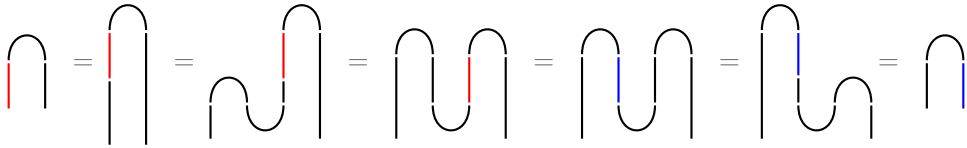

\begin{cor}
	For a $T^\b$-object $\O$ and any $A\in\SL$ we have
\begin{align}
	\b^\O\circ(\rho^\O_A\otm \mj) = \b^\O\circ(\mj\otm \rho^\O_{JA^{-1}J}) \label{eq:Tb-object-beta-jump}\\
	(\rho^\O_A\otm \mj)\circ\g^\O = (\mj\otm \rho^\O_{JA^{-1}J})\circ\g^\O.\label{eq:Tb-object-gamma-jump}
\end{align}
\end{cor}
\begin{proof}
	This follows from Lemma~\ref{lem:Tb-variation}, using the equality $(f\circ g)\otm \mj= (f\otm \mj)\circ (g\otm \mj)$, the property~\eqref{eq:Tb-object} and its variants, and the fact from Proposition~\ref{prop:SL-presentation} that $D_a$ and $D_b=JD_a^{-1}J$ generate $\SL$.
\end{proof}

The following will be a key fact later, and should be compared to Corollary~\ref{cor:Bun-conj}.
\begin{prop}
\label{prop:classify-new-proof}
	For a $T^\b$-object $\O$ and any $A,C\in\SL$ we have
\begin{equation}
\label{eq:Tb-object-key}
	\b^\O\circ(\rho^\O_{CA^{-1}C}\otm\mj)\circ\g^\O=\b^\O\circ(\rho^\O_A\otm \mj)\circ\g^\O = \b^\O\circ(\rho^\O_{JA^{-1}J}\otm\mj)\circ\g^\O.
\end{equation}
\end{prop}
\begin{proof}
Firstly, by the properties~\eqref{eq:tau-beta} and~\eqref{eq:tau-gamma} of the $T^\b$-object $\O$ in $\calC$ we have
\begin{align*}
	\b^\O\circ(\rho^\O_A\otm \mj)\circ\g^\O
	& = \b^\O\circ(\rho^\O_A\otm \mj)\circ\tau \circ\g^\O\\
    & = \b^\O\circ\tau\circ(\mj\otm \rho^\O_A)\circ\g^\O = \b^\O\circ(\mj\otm \rho^\O_A)\circ\g^\O.
\end{align*}
Using this and also the property~\eqref{eq:Tb-object-beta-jump} we deduce
\begin{align*}
	\b^\O\circ (\rho^\O_A\otm\mj)\circ\g^\O
		&=\b^\O\circ (\mj\otm\rho^\O_{JA^{-1}J})\circ\g^\O
		=\b^\O\circ (\rho^\O_{JA^{-1}J}\otm\mj)\circ\g^\O.
\end{align*}
Similarly, for $C\in\SL$ using~\eqref{eq:Tb-object-beta-jump} and~\eqref{eq:Tb-object-gamma-jump} we find
	\begin{align*}
	\b^\O\circ (\rho^\O_{CAC^{-1}}\otm\mj)\circ\g
		&=\b^\O\circ (\rho^\O_{CA}\otm\mj)\circ(\rho^\O_{C^{-1}}\otm\mj)\circ\g^\O\\
		&=\b^\O\circ (\rho^\O_{CA}\otm\mj)\circ(\mj\otm\rho^\O_{JCJ})\circ\g^\O\\
		&=\b^\O\circ (\mj\otm\rho^\O_{JCJ})\circ (\rho^\O_{CA}\otm\mj)\circ\g^\O\\
		&=\b^\O\circ (\rho^\O_{C^{-1}}\otm\mj)\circ (\rho^\O_{CA}\otm\mj)\circ\g^\O\\
		&=\b^\O\circ (\rho^\O_{C^{-1}CA}\otm\mj)\circ\g^\O\\
		&=\b^\O\circ (\rho^\O_A\otm\mj)\circ\g^\O.\qedhere
	\end{align*}
\end{proof}

%The naturality of symmetry implies the following result.
%\begin{lem}
%	For any $T^\b$-object $\O$ and any $f,g:\O\otm\O\to\O\otm\O$ we have
%	\begin{equation}
%	\label{eq:permuTbe-Bun}
%		(\b^\O\circ f\circ\g^\O)\otm (\b^\O\circ g\circ\g^\O)= (\b^\O\circ g\circ\g^\O)\otm (\b^\O\circ f\circ\g^\O).
%	\end{equation}
%\end{lem}

%%%%%%%%%%
\subsection{\texorpdfstring{$T^\be$}{Tbe}-objects}
\label{subsec:Tbe-objects}

\begin{dfn}
\label{defin:TbeObject}
Fix a symmetric, strict monoidal category $\calC$, with the unit object $\1$, the symmetry $\tau$, and the identity arrows $\mj$. A $T^\be$-\emph{object} in $\calC$ is an object $\O$ that is a $T^\b$-object as in Definition~\ref{dfn:Tb-object} with the following additional data.
\begin{itemize}
	\item There is a \emph{unit} $\unit^\O:\1\to \O$ and a \emph{counit} $\counit^\O:\O\to \1$ satisfying
\begin{equation}\label{eq:unit-counit}
	\counit^\O = \b^\O\circ (\mj\otm \unit^\O).
\end{equation}
	\item Moreover, for $D_a\in\SL$ from~\eqref{eq:Da,Db} { and for $-E\in\SL$} the following { equalities hold}
\begin{align}
	\rho^\O_{D_a}\circ \unit^\O &= \unit^\O, \label{eq:Da-e}\\
	{ \rho^\O_{-E}\circ \unit^\O} &{ = \unit^\O.}\label{eq:-E-e} 
\end{align}
\end{itemize}
\end{dfn}

\begin{rem}\label{rem:beta-e,eta-gamma}
	It is straightforward to conclude that for every $T^\be$-object we have $\b^\O\circ (\unit^\O\otm \mj)= \counit^\O$ and $(\mj\otm \counit^\O)\circ\g^\O=\unit^\O=(\counit^\O\otm\mj)\circ\g^\O$.
\end{rem}

\begin{lem}
We have
\begin{equation}\label{eq:eta-Db}
	\counit^\O\circ\rho^\O_{D_b}=\counit^\O.
\end{equation}
\end{lem}
\begin{proof}
	Using~\eqref{eq:unit-counit}, ~\eqref{eq:Tb-object} and~\eqref{eq:Da-e} we have
		\begin{align*}
		\counit^\O\circ\rho^\O_{D_b}
		=\b^\O\circ(\mj\otm\unit^\O) \circ\rho^\O_{D_b}
		&= \b^\O\circ(\rho^\O_{D_b}\otm\mj)\circ(\mj\otm\unit^\O)\\
		&=\b^\O\circ(\mj\otm\rho^\O_{D_a})\circ(\mj\otm\unit^\O)\\
		&=\b^\O\circ(\mj\otm(\rho^\O_{D_a}\circ\unit^\O))\\
		&=\b^\O\circ(\mj\otm\unit^\O)=\counit^\O.\qedhere
		\end{align*}
\end{proof}

{ The following is proved in an analogous manner using~\eqref{eq:-E-e} instead of~\eqref{eq:Da-e}.

\begin{lem}
We have
\begin{equation}\label{eq:eta--E}
	\counit^\O\circ\rho^\O_{-E}=\counit^\O.
\end{equation}
\end{lem}}

\begin{prop}\label{prop:abstr-lens-space}
	For a $T^\be$-object $\O$, if $A,A'\in\SL$ are lens-inseparable then
\begin{equation}
	\counit^\O\circ\rho^\O_A\circ \unit^\O=\counit^\O\circ\rho^\O_{A'}\circ \unit^\O.
\end{equation}
\end{prop}
\begin{proof}
	It is enough to prove that the equality holds for 1-step inseparable matrices. If $A'=D_b^nAD_a^k$ it does hold, thanks to~\eqref{eq:Da-e} and~\eqref{eq:eta-Db}. If $A'=JA^{-1}J$ then
\begin{equation*}
\begin{split}
	\counit^\O\circ\rho^\O_{JA^{-1}J}\circ \unit^\O
	& = \b^\O\circ(\mj\otm \unit^\O)\circ \rho^\O_{JA^{-1}J}\circ\unit^\O
\\
	&  = \b^\O\circ(\rho^\O_{JA^{-1}J}\otm \mj)\circ(\unit^\O\otm \unit^\O)
\\
	&  = \b^\O\circ(\mj\otm \rho^\O_A)\circ(\unit^\O\otm \unit^\O)
\\
	&   = \b^\O\circ(\unit^\O\otm \mj) \circ \rho^\O_A\circ\unit^\O
\\
	&  = \counit^\O\circ \rho^\O_A \circ\unit^\O,
\end{split}
\end{equation*}
	using~\eqref{eq:unit-counit}, the property~\eqref{eq:Tb-object-beta-jump} of $T^\b$-objects, and Remark~\ref{rem:beta-e,eta-gamma}.
\end{proof}

%\begin{dfn}
% An \emph{abstract lens space} in $\calC$ with respect to a $T^\be$-object is a arrow in $\calC$ given as a composite $\counit\circ\rho^\O_A\circ \unit$ for some $A\in\SL$. A \emph{concrete lens space} is an abstract lens space in the category $\beMCG$.
%\end{dfn}

\subsection{Universality}
\label{subsec:universality}

\begin{lem}\label{lem:T-Tb}
\begin{enumerate}
	\item The torus $T\in\bMCG$ is a $T^\b$-object.
	\item The torus $T\in\beMCG$ is a $T^\be$-object.
\end{enumerate}
\end{lem}
\begin{proof}
	For $\O=T\in\bMCG$ we set $\rho^\O_A\coloneqq\Cyl_A$, $\b^\O\coloneqq\b$, $\g^\O\coloneqq\g$, as in~\eqref{eq:beta} --~\eqref{eq:Cyl}. Then~\eqref{eq:tau-beta} and~\eqref{eq:tau-gamma} hold by Lemma~\ref{lem:tau}, and~\eqref{eq:snake} by Lemma~\ref{lem:snake}. We have
\[
	\b\circ (\Cyl_{D_a}\otm\mj)=\b\circ (\mj\otm\Cyl_{JD_a^{-1}J}),
\]
	by Lemma~\ref{lem:beta-jump}, and since $JD_a^{-1}J=D_b$ by~\eqref{eq:J}, the desired property~\eqref{eq:Tb-object} holds.

	For the second claim, note that $T\in\beMCG$ is a $T^\b$-object. { It remains to verify~\eqref{eq:unit-counit},~\eqref{eq:Da-e} and~\eqref{eq:-E-e}, which hold by Lemma~\ref{lem:e-eta}, Lemma~\ref{lem:Da-e} and Lemma~\ref{lem:-E-e} respectively.}
\end{proof}

{ Our next goal is to define a category that provides a ``universal'' $T^\b$-object, in the same sense that the simplicial category $\Delta$ provides a ``universal'' monoid (see \cite[Proposition~1, Section VII.5]{ML98}). To this end, consider the category $\mathcal{A}^\b$ whose objects are pairs $(\calC,\O)$, where $\calC$ is a symmetric strict monoidal category and $\O$ is a $T^\b$-object in $\calC$. The arrows of $\mathcal{A}^\b$ from $(\calC,\O)$ to $(\calC',\O')$ are symmetric strict monoidal functors from $\calC$ to $\calC'$ that send $\O$ to $\O'$ and preserve the structure of $T^\b$-objects ``on the nose''.

Note that everything in this setting is equationally presented. Thus, the forgetful functor to the category of sets, sending $(\calC,\O)$ to the set of objects of $\calC$, has a left adjoint, and the following makes sense.
\begin{dfn}\label{sloboda}
	Let $\calC^\b$ be an initial object of $\mathcal{A}^\b$.
	
	Analogously, let $\calC^\be$ be an initial object of the category whose objects are pairs $(\calC,\O)$, for a symmetric strict monoidal category $\calC$ and a $T^\be$-object $\O$ in $\calC$.
\end{dfn}
A brief construction of $\calC^\b$ and $\calC^\be$ is given in the first paragraph of Appendix~\ref{sec:appendix}.
}

The main result about $T^\b$- and $T^\be$-objects is the following.
\begin{thm}\label{thm:freely}
	There are isomorphisms
	\[
		F^\b:\calC^\b\ra \bMCG, \quad F^\be:\calC^\be\ra \beMCG,
	\]
	that preserve the structures of symmetric strict monoidal categories with $T^\b$- and $T^\be$-objects respectively.
	
	Moreover, any arrow in $\bMCG$ is up to symmetries represented as the tensor product of arrows of the form (with $A\in\SL$ varying):
\[
	C_1\coloneqq\beta^\O\circ (\rho^\O_A\otm\mj), \quad
	C_2\coloneqq\rho^\O_A, \quad
	C_3\coloneqq(\rho^\O_A\otm\mj)\circ \gamma^\O, \quad
	C_4\coloneqq\beta^\O\circ (\rho^\O_A\otm\mj)\circ \gamma^\O.
\]
	For $\beMCG$ we additionally have { for $A=\matrica{ p & r \\ q & s }\in\SL$, with $p\geq 0$:}
\[
	C_0\coloneqq\counit^\O\circ\rho_A^\O\circ\unit^\O, \qquad
	C_5\coloneqq\rho_A^\O\circ\unit^\O, \qquad
	C_6\coloneqq\counit^\O\circ\rho_A^\O.
\]
\end{thm}

The proof is not hard but is lengthy and is deferred to Appendix~\ref{sec:appendix}.
\begin{cor}\label{thm:Tb-object}
\hfill
\begin{enumerate}
\item
	The torus $T\in\bMCG$ is the universal $T^\b$-object. That is, for any $T^\b$-object $\O\in\calC$ there is a unique symmetric monoidal functor $F:\bMCG \to \calC$, with
	\[
		F(T)=\O,\quad F(\Cyl_A)=\rho^\O_A,\quad F(\b)=\b^\O,\quad F(\g)=\g^\O.
	\]	
\item
	The torus $T\in\beMCG$ is the universal $T^\be$-object. That is, for a $T^\be$-object $\O\in\calC$ there is a unique symmetric monoidal functor $F:\beMCG\, \to \calC$, with the same conditions as in (i) together with
	\[
		F(\unit)=\unit^\O.
	\]
\end{enumerate}
\end{cor}
%\begin{proof}
%	We argue in the first case, while the second is similar. The functor $F$ is determined by the listed values on generators, by extending with respect to $\circ, \otm, \tau$. It remains to show that it is well defined. In other words, we need to show the following: if an arrow of $\bMCG$ is written in two different ways in terms of the generators, then this equality holds in $\calC$, that is, it is implied by the given axioms~\eqref{eq:tau-beta} --~\eqref{eq:Tb-object}. This follows from Theorem~\ref{thm:freely}.
%\end{proof}

We also record the following useful consequences of Theorem~\ref{thm:freely}.
\begin{cor}\label{cor:support-b}
	The arrows in the category $\bMCG$ are precisely those cobordisms in $\Cob$ that are supported by 3-manifolds that are disjoint unions of $T\tm I$ and torus bundles over $S^1$.
\end{cor}
\begin{proof}
	% The categories $\bMCG$ and $\bMCG^\star$  have the same objects: in the former they are sequences of $\torus$, and in the latter they are sequences of equivalences classes of the surface homeomorphic to $\torus$ (see the opening paragraph of the paper).
	% Thus, we need only show that any arrow in $\bMCG^\star$ is represented as a composite of arrows of $\bMCG$. For this
	In one direction, each cobordism supported either by the 3-manifold $T\tm I$ or a $T$-bundle over $S^1$ is in $\bMCG$. The former case is the content of Lemma~\ref{lem:cob-supp-by-TxI}, whereas the latter is covered by Proposition~\ref{prop:Bun}.
	
	In the other direction, by Theorem~\ref{thm:freely} any arrow in $\bMCG$ is up to symmetries a tensor product of the arrows $C_1$ -- $C_4$. Since the first three are of the shape $T\tm I$ and the last one is a $T$-bundle, we conclude that any arrow is of desired form.
\end{proof}
	
\begin{cor}\label{cor:support-be}
		The arrows in the category $\beMCG$ are precisely those cobordisms in $\Cob$ that are supported by 3-manifolds that are disjoint unions of $T\tm I$, $D^2\tm S^1$, torus bundles over $S^1$, and lens spaces.
\end{cor}
\begin{proof}
	% The categories $\bMCG$ and $\bMCG^\star$  have the same objects: in the former they are sequences of $\torus$, and in the latter they are sequences of equivalences classes of the surface homeomorphic to $\torus$ (see the opening paragraph of the paper).
	% Thus, we need only show that any arrow in $\bMCG^\star$ is represented as a composite of arrows of $\bMCG$. For this
	In one direction, we need to show that each of the listed 3-manifolds is in $\beMCG$. Using Corollary~\ref{cor:support-b}, it suffices to consider the case $D^2\tm S^1$. Indeed, this supports cobordisms either of the shape $\iota\circ A: T\hra D^2\tm S^1\hla\emptyset$ for $A\in\SL$, which is equal to $\Cyl_{A^{-1}}\circ\unit$, or $\emptyset\hra D^2\tm S^1\hla T:\iota\circ A'$ for $A'\in\GL\setminus\SL$, which is equal to $\counit\circ\Cyl_{JA'}$ (see Corollaries~\ref{cor:collar1} and~\ref{cor:collar2}).
	
	In the other direction, by Theorem~\ref{thm:freely} any arrow in $\beMCG$ is up to symmetries the tensor product of the arrows of the type $C_0$ -- $C_6$. We have that $C_0$ is a lens space, each of $C_1$, $C_2$, $C_3$ is supported by $T\tm I$, while $C_4$ is a torus bundle, and $C_5$, $C_6$ are supported by $D^2\tm S^1$. Thus, any arrow is of desired form.
\end{proof}

%%%%%%%%
\section{Some restricted (2+1) TQFTs from low-dimensional \texorpdfstring{$SL(2,\Z)$}{SL(2,Z)} representations}
\label{sec:tqft}

In this section we construct restricted TQFTs that distinguish some pairs of torus bundles and some interesting lens spaces. 

More precisely, we construct symmetric monoidal functors $\bMCG\to\Vec_{\KK}$ and $\beMCG\to\Vec_{\KK}$ that give different values on the mentioned 3-manifolds. Here $\Vec_{\KK}$ is the skeleton of the category of finite dimensional vector spaces over $\KK=\RR$ or $\CC$: its objects are $\KK^n$ for $n\geq1$ and arrows are matrices with respect to fixed bases. We use the following reformulation of Corollary~\ref{thm:Tb-object}.

\newpage
\begin{cor}\label{cor:TQFT}\hfill
	\begin{itemize}
		\item A restricted TQFT $F:\bMCG\to\Vec_{\KK}$ is completely described by the data:
	\begin{align*}
		& n_F\coloneqq\dim(F(T))\in\Z_{\geq1},\\
		& \rho^F_a\coloneqq F(\Cyl_{D_a})\in GL(n_F,\KK),\;\rho^F_b\coloneqq F(\Cyl_{D_b})\in GL(n_F,\KK),\\
		& \b^F\coloneqq F(\b)\in \Mat(2n_F \tm 1,\KK),\;\g^F\coloneqq F(\g)\in \Mat(1\tm 2n_F,\KK),
	\end{align*}
	that satisfy:
	\begin{align}
		\rho^F_a\rho^F_b\rho^F_a&=\rho^F_b\rho^F_a\rho^F_b,\label{eq-rep1}\\
		(\rho^F_a\rho^F_b)^6&=1,\label{eq-rep2}\\
		\b^F\tau&=\b^F,\label{eq-b} \\
		\tau\g^F&=\g^F,\label{eq-g} \\
		(\b^F\otm E_{n_F})(E_{n_F}\otm\g^F)&=E_{n_F}=(E_{n_F}\otm\beta^F)(\g^F\otm E_{n_F}),\label{eq-bg}\\
		\b^F(\rho^F_a\otm E_{n_F}) &= \b^F(E_{n_F}\otm \rho^F_b).\label{eq-rep-b}
	\end{align}
	for the identity matrix $E_{n_F}\in GL(n_F,\KK)$ and $\tau=\matrica{0_{n_F} & E_{n_F}\\ E_{n_F} & 0_{n_F}}\in GL(2n_F,\KK)$.
		\item A restricted TQFT $F:\beMCG\to\Vec_{\KK}$ is completely described by the same data as above together with
	\begin{align*}
		& \unit^F\coloneqq F(\unit)\in\Mat(1\tm n_F,\KK),
	\end{align*}
	that is an eigenvector of { matrices} $\rho^F_a$ { and $\rho^F_{-E}:=F(\Cyl_{-E})$} with eigenvalue $1$, that is:
	\begin{align}
		\rho^F_a\circ \unit^F &= \unit^F,\label{eq-e}\\
		{\rho^F_{-E}\circ\unit^F}&={\unit^F}.
	\end{align}	
\end{itemize}
\end{cor}

Moreover, recall from Proposition~\ref{prop:Bun} that $Bun_A\cong\b\circ(\Cyl_A\otm\mj)\circ\g$. Thus, we have
\[
	F(Bun_A)=\b^F\circ(\rho^F_A\otm E_{n_F})\otm\g^F.
\]

\begin{lem}\label{lem:trace}
	For a restricted TQFT $F:\bMCG\to\Vec_{\KK}$ and a matrix $A\in\SL$ we have
	\[
		F(Bun_A)=\tr(\rho^F_A).
	\]
\end{lem}
\begin{proof}
	The proof is probably well known, but we include it for completeness.
	
	We can assume that $\g^F(1)=\sum_ie_i\otm v_i$ for some $v_i\in F(T)$, $1\leq i\leq n$, and the standard orthonormal basis $(e_1,\dots,e_n)$ of $F(T)\cong\KK^n$. Then by definition $F(Bun_A)$ sends $1\in\KK$ to
	\begin{align*}
		1 	\overset{\g^F}{\mapsto} \sum_ie_i\otm v_i
			\overset{\rho^F_A\otm E_{n_F}}{\mapsto} \sum_i \rho^F_A(e_i)\otm v_i
			\overset{\b^F}{\mapsto} & \sum_i \b^F(\rho^F_A(e_i), v_i).
	\end{align*}
	By~\eqref{eq-bg} and~\eqref{eq-g} for any $v\in F(T)$ we have $v=\sum_i\b^F(v, v_i)e_i$. Thus, putting $v=\rho^F_A(e_j)$ gives $\rho^F_A(e_j)=\sum_i \b^F(\rho^F_A(e_j), v_i) e_i$.
	Applying $\sum_j\langle -,e_j\rangle$ to both sides we obtain the desired equality
	\[
	\tr(\rho^F_A) =\sum_j\langle\rho^F_A(e_j),e_j\rangle
	=\sum_{i,j}\b^F(\rho^F_A(e_j), v_i)\langle e_i,e_j\rangle
	=\sum_i\b^F(\rho^F_A(e_i), v_i).\qedhere
	\]
\end{proof}

Recall from Proposition~\ref{prop:SL-presentation} that for any matrix $A=\matrica{ p & r \\ q & s }\in\SL$ and $\frac{p}{q}=m_1-\tfrac{1}{m_2-\frac{1}{\dots-\frac{1}{m_k}}}$ we have $A=\mathrm{sgn}(p) D_a^{m_1-1}D_b^{-1}D_a^{m_2-2}\cdots D_b^{-1} D_a^{m_k-2}D_b^{-1}D_a^m$, for a uniquely determined $m\in\Z$.
Therefore, for $\rho^F_A\coloneqq F(\Cyl_A)$ we have
	\begin{equation}\label{eq:A-final}
	\rho^F_A=
		(\rho^F_{\mathrm{sgn}(p)E}) (\rho^F_a)^{m_1-1}(\rho^F_b)^{-1}(\rho^F_a)^{m_2-2}\cdots (\rho^F_b)^{-1} (\rho^F_a)^{m_k-2}(\rho^F_b)^{m-1}.
	\end{equation}

\subsection{Distinguishing some Turaev--Viro-equivalent torus bundles}\label{subsec:dist-TV}

\begin{thm}[Funar~\cite{F13}]\label{thm:Funar1}
	For any $k\in\Z\setminus\{0\}$, and a prime $q\in\Z$ with $q\equiv1\pmod{4}$, and $v\in\Z_{>0}$ such that $-v\pmod{q}$ is a non-zero square, and $v$ is divisible either by $4$ or by a prime $p$ satisfying $p\equiv3\pmod{4}$, consider the pairs of matrices
\[
	G_{k,q,v}=\matrica{1 & kq^2\\kv & 1+k^2q^2v},\hskip4mm H_{k,q,v}=\matrica{1 & k\\kq^2v & 1+k^2q^2v}.
\]
	The associated 3-manifolds $Bun_{G_{k,q,v}}$ and $Bun_{H_{k,q,v}}$ are not homeomorphic, but all of their Turaev--Viro invariants agree: for any spherical fusion category $\calC$ we have
	\[
		TV_{\calC}(Bun_{G_{k,q,v}})=TV_{\calC}(Bun_{H_{k,q,v}}).
	\]
\end{thm}
\begin{rem}\label{rem:Funar1}
	In fact, Funar~\cite[Proposition~1.1]{F13} shows that if $G,H\in\SL$ are conjugate in every $SL(2,\Z/m\Z)$, then $Bun_G$ and $Bun_H$ have the same Turaev--Viro invariants; in~\cite[Proposition~1.3]{F13} he shows that this property holds for the above family of pairs. Stebe~\cite{Stebe} has shown the same holds true for
\[
   G=\matrica{188 & 275\\121 & 177},\hskip4mm H=\matrica{188 & 11\\3025 & 177}.
\]	
\end{rem}

We define a restricted TQFT that does distinguish all of the above 3-manifolds.
\begin{thm}\label{thm:F1}
	The restricted TQFT $F_1:\bMCG\to\Vec_{\RR}$ defined by
\begin{align*}
	&n_F=2,\;
	\rho^{F_1}_a=\frac{1}{\sqrt{2}}\matrica{1 & 1\\0 & 2},\;
	\rho^{F_1}_b=\frac{1}{\sqrt{2}}\matrica{2 & 0\\-2 & 1},\\
	&\b^{F_1}=\begin{pmatrix}2 & 1 & 1 & -1\end{pmatrix},\;
	\g^{F_1}=\frac{1}{3}\begin{pmatrix}1 & 1 & 1 & -2\end{pmatrix}^T,
\end{align*}

	satisfies $F_1(Bun_{G_{k,q,v}})\neq F_1(Bun_{H_{k,q,v}})$ for all triples $(k,q,v)$ from Theorem~\ref{thm:Funar1}.
\end{thm}

\begin{proof}
The 2-dimensional representation $\rho^{F_1}$ is obtained from \cite[Prop.2.5, p.504]{Tuba-Wenzl} by putting $\lambda_1=1$, $\lambda_2=2$. Thus, we know that the conditions~\eqref{eq-rep1} and~\eqref{eq-rep2} from Corollary~\ref{cor:TQFT} are fulfilled. Then, we search for $\b$ that satisfies~\eqref{eq-b} and~\eqref{eq-rep-b}, and pick one possible solution of the resulting system of equations. Finally, we choose $\g$ that satisfies~\eqref{eq-g} and~\eqref{eq-bg}.
	
	Let us compute the resulting invariants.
For $m\in\N$ we first compute
	\[
	(\rho^{F_1}_a)^m=\frac{1}{2^\frac{m}{2}}\matrica{1 & 2^m-1\\0 & 2^m},
	\hskip4mm
	(\rho^{F_1}_b)^m=\frac{1}{2^\frac{m}{2}}\matrica{2^m & 0\\-2(2^m-1) & 1},
	\]
and their inverses are
	\[
	(\rho^{F_1}_a)^{-m}=\frac{1}{2^\frac{m}{2}}\matrica{2^m & 1-2^m\\0 & 1},
	\hskip4mm
	(\rho^{F_1}_b)^{-m}=\frac{1}{2^\frac{m}{2}}\matrica{1 & 0\\2(2^m-1) & 2^m}.
	\]
Then we use~\eqref{eq:A-final} to find $G_{k,q,v}=D_b^{-kv}D_a^{kq^2}$ and $H_{k,q,v}=D_b^{-kq^2v}D_a^{k}$.
When $k>0$ we compute
	\[
	\rho^{F_1}_{G_{k,q,v}}=\frac{1}{2^\frac{kv+kq^2}{2}}
				\matrica{1 & 2^{kq^2}-1\\2(2^{kv}-1) & 2(2^{kv}-1)(2^{kq^2}-1)+2^{kv+kq^2}},
	\]
and
	\[
	\rho^{F_1}_{H_{k,q,v}}=\frac{1}{2^\frac{kq^2v+k}{2}}
				\matrica{1 & 2^{k}-1\\2(2^{kq^2v}-1) & 2(2^{kq^2v}-1)(2^{k}-1)+2^{kq^2v+k}},
	\]
whereas for $k<0$ we get
	\[
	\rho^{F_1}_{G_{k,q,v}}=2^\frac{kv+kq^2}{2}
				\matrica{2^{-k(v+q^2)} & 2^{-kv}(1-2^{-kq^2})\\-2^{-kq^2+1}(2^{-kv}-1) & 2(2^{-kv}-1)(2^{-kq^2}-1)+1},
	\]
	and
	\[
	\rho^{F_1}_{H_{k,q,v}}=2^\frac{kq^2v+k}{2}
				\matrica{2^{-k(q^2v+1)} & 2^{-kq^2v}(1-2^k)\\-2^{-k+1}(2^{-kq^2v}-1) & 2(2^{-kq^2v}-1)(2^{-k}-1)+1}.
	\]
The traces are given by
\begin{align*}
	&\tr(\rho^{F_1}_{G_{k,q,v}})=2^{-\frac{|k|(v+q^2)}{2}}\big(1+2(2^{|k|v}-1)(2^{|k|q^2}-1)+2^{|k|(v+q^2)}\big)\\
	&\tr(\rho^{F_1}_{H_{k,q,v}})=2^{-\frac{|k|(q^2v+1)}{2}}\big(1+2(2^{|k|q^2v}-1)(2^{|k|}-1)+2^{|k|(q^2v+1)}\big).
\end{align*}
	In both cases the number in the bracket is odd, so these traces can agree only if their exponents of 2 agree, that is, if and only if $-\frac{|k|(v+q^2)}{2}=-\frac{|k|(q^2v+1)}{2}$, or equivalently $v-1=q^2(v-1)$. Since positive integers $q$ and $v$ are greater than 1, we conclude that these traces are always different. By Lemma~\ref{lem:trace} we conclude that
	\[
		F_1(Bun_{G_{k,q,v}})=\tr(\rho^{F_1}_{G_{k,q,v}})\neq \tr(\rho^{F_1}_{H_{k,q,v}})=F_1(Bun_{H_{k,q,v}}).
	\]
That is, the TQFT $F_1$ on $\bMCG$ distinguishes each $Bun_{G_{k,q,v}}$ from $Bun_{H_{k,q,v}}$.
\end{proof}

\begin{rem}
	As $1$ is not an eigenvalue of $\rho^{F_1}_a$ we cannot extend $F_1$ to $\beMCG$.
\end{rem}

\begin{rem}
	The work~\cite{Tuba-Wenzl} classifies all 2-dimensional representations of $\SL$.
\end{rem}

\subsection{{ Distinguishing many torus bundles and lens spaces simultaneously}}\label{subsec:dist-lens}
In this section we { consider} several interesting pairs of lens spaces. 

We start with the simplest pair $L(7,1)$ and $L(7,2)$ of homotopy equivalent, but not homeomorphic lens spaces. This pair is distinguishable by Turaev--Viro invariants, as shown in \cite[Proposition~2]{Sokolov}. Then we investigate the pair $L(65,8)$ and $L(65,18)$, not distinguished by Reshetikhin--Turaev invariants by Remark 3.9 in~\cite{Jeffrey}. See also \cite{Constatino} where the same pair is studied. Finally, we study the pairs $L(11,2)$, $L(11,4)$ and $L(13,3)$, $L(13,5)$ which are not distinguishable by Turaev--Viro invariants (see \cite[Corollary~4]{Sokolov}).

{ Furthermore, we consider the following torus bundles.}
\begin{thm}[Funar~\cite{F13}]\label{thm:Funar2}
	There exist infinitely many pairs $K_n,L_n\in\SL$, $n\geq1$ such that the associated 3-manifolds $Bun_{K_n}$ and $Bun_{L_n}$ are not homeomorphic, but all of their Reshetikhin--Turaev invariants agree: for any modular tensor category $\calC$ and $n\geq1$ we have
	\[
		RT_{\calC}(Bun_{K_n})=RT_{\calC}(Bun_{L_n}).
	\]
	In particular, the same holds for the pairs of matrices
	\begin{align*}
    &X_{21}=\matrica{1 & 21\\21 & 442},\hskip5.5mm  Y_{21}=\matrica{106 & 189\\189 & 337},\\
    &X_{51}=\matrica{1 & 51\\51 & 2602},\hskip4mm Y_{51}=\matrica{562 & 1071\\1071 & 2041},\\ %not 1071
    &X_{53}=\matrica{1 & 53\\53 & 2810},\hskip4mm Y_{53}=\matrica{425 & 1007\\1007 & 2386},\\
    &X_{55}=\matrica{1 & 55\\55 & 3026},\hskip4mm Y_{55}=\matrica{881 & 1375\\1375 & 2146}.
	\end{align*}
\end{thm}
\begin{rem}\label{rem:Funar2}
	There is a misprint in $Y_{51}$ in the cited work~\cite{F13}.
	
	In fact, Funar~\cite[Proposition~1.1]{F13} shows that if $K,L\in\SL$ are conjugate in every $SL(2,\Z/m\Z)$, and have the same modified Rademacher function $\varphi(K)=\varphi(L)$, then $Bun_K$ and $Bun_L$ have the same Reshetikhin--Turaev invariants. The pairs $G_{k,q,v},H_{k,q,v}$ from Theorem~\ref{thm:Funar1} have distinct $\varphi$-values, and are in fact distinguished by Reshetikhin--Turaev invariants. The infinite family $K_n,L_n$ with desired properties does exist but is not explicit; see \cite[p.~2320]{F13} and references mentioned there.
\end{rem}

{ Since linear two-dimensional TQFTs on $\beMCG$ fail to distinguish pairs of lens spaces, we next consider three-dimensional representations of $\SL$.}
\begin{thm}\label{thm:F3}
	The restricted TQFT ${ F_2}:\beMCG\to\Vec_{\RR}$ defined by
\begin{align*}
	&n_{F_2}=3,\quad
	\rho^{F_2}_a=\frac{1}{2}\begin{pmatrix}1 & 4 & 2\\0 & 2 & 2\\ 0 & 0 & 4\end{pmatrix},\quad
	\rho^{F_2}_b=\frac{1}{2}\begin{pmatrix}4 & 0 & 0\\-2 & 2 & 0\\2 & -4 & 1\end{pmatrix},\\
	&\b^{F_2}=\begin{pmatrix}2 & 4 & 1 & 4 & -4 & -4 & 1 & -4 & 2\end{pmatrix},\quad
	\g^{F_2}=\frac{1}{36}\begin{pmatrix}8 & 4 & 4 & 4 & -1 & -4 & 4 & -4 & 8\end{pmatrix}^T,\\
	&\counit^{F_2}=\begin{pmatrix}12 & 12 & 0\end{pmatrix},\quad
	\unit^{F_2}=\begin{pmatrix}4 & 1 & 0 \end{pmatrix}^T,
\end{align*}
	satisfies
	\[F_2(Bun_{G_{k,q,v}})\neq F_2(Bun_{H_{k,q,v}})\]
	for all triples $(k,q,v)$ from Theorem~\ref{thm:Funar1},
	and also
	\[F_2(Bun_G)\neq F_2(Bun_H)\]
	for $G,H$ from Remark~\ref{rem:Funar1},
	and also
	\[F_2(Bun_{X_i})\neq F_2(Bun_{Y_i})\]
	for $i=21,51,53,55$ from Theorem~\ref{thm:Funar2},
	as well as
	\[\begin{array}{l}F_2(L(7,1))\neq F_2(L(7,2)), \quad\quad F_2(L(65,8))\neq F_2(L(65,18)),\\[2ex]
    F_2(L(11,2))\neq F_2(L(11,4)), \quad F_2(L(13,3))\neq F_2(L(13,5)).\end{array}
    \]
\end{thm}

\begin{proof}
	The 3-dimensional representation $\rho^{F_2}$ is obtained from \cite[Prop.2.5, p.504]{Tuba-Wenzl} by putting $\lambda_1=1$, $\lambda_2=2$, $\lambda_3=4$. We check that all assumptions from Corollary~\ref{cor:TQFT} are fulfilled by a straightforward computation. For example, for~\eqref{eq-rep-b} we have
	\[
	\b^{F_2}\circ (\rho^{F_2}_a\otm E_3) = \begin{pmatrix} 1 & 2 & \frac{1}{2} & 8 & 4 & -2 & 8 & -8 & 1 \end{pmatrix}
	= \b^{F_2}\circ (E_3\otm \rho^{F_2}_b),
	\]
whereas for~\eqref{eq-e} we check that $\unit^{F_2}$ is an eigenvector of $\rho^{F_2}_a$ for the eigenvalue $1$.

	Let us compute the resulting invariants.
We find
\begin{align*}
    \tr(\rho^{F_2}_{G_{k,q,v}})
    	=2^{-|k|(v+q^2)}\big(1&+4(2^{|k|q^2}-1)(2^{|k|v}-1)+\\
    	&+2^{|k|(v+q^2)}+4(2^{|k|q^2}-1)^2(2^{|k|v}-1)^2+\\
    	&+2^{2+|k|(v+q^2)}(2^{|k|q^2}-1)(2^{|k|v}-1)+4^{|k|(v+q^2)}\big)
\end{align*}
and
\begin{align*}
    \tr(\rho^{F_2}_{H_{k,q,v}})
    	=2^{-|k|(q^2v+1)}\big(1&+4(2^{|k|q^2v}-1)(2^{|k|}-1)+\\
    	&+2^{|k|(q^2v+1)}+4(2^{|k|q^2v}-1)^2(2^{|k|}-1)^2+\\
    	&+2^{2+|k|(q^2v+1)}(2^{|k|q^2v}-1)(2^{|k|}-1)+4^{|k|(q^2v+1)}\big).
\end{align*}
Arguing as in the proof of Theorem~\ref{thm:F1} we see that these traces are always different. Using Lemma~\ref{lem:trace} we conclude that $F_2$ distinguishes each $Bun_{G_{k,q,v}}$ from $Bun_{H_{k,q,v}}$.

Next, for the pair $G=\matrica{188 & 275\\121 & 177}$ and $H=\matrica{188 & 11\\3025 & 177}$ we use Wolfram Mathematica to find
\[
	\tr(\rho^{F_2}_G)=\frac{k_1}{2^{16}}\neq \frac{k_2}{2^{44}}=\tr(\rho^{F_2}_H)
\]
for some odd integers $k_1,k_2$, implying that these traces are distinct.

Similarly, for $i=21,51,53,55$ we compute:
\begin{align*}
	&\tr(\rho^{F_2}_{X_i})-\tr(\rho^{F_2}_{Y_i})=\frac{l_i}{2^{2i}},
\end{align*}
where $l_i$ are certain large numbers that we omit.

{ In the category $\beMCG$ we have that
	\begin{align*}
		L(7,i)=\counit\circ\Cyl_{\Lambda_i}\circ \unit
	\quad\text{ for }
		\Lambda_1=\matrica{7 & -8\\1 & -1},\quad
		\Lambda_2=\matrica{7 & -4\\2 & -1}.
	\end{align*}
	Using Proposition~\ref{prop:SL-presentation} we have decompositions
	\[
		\Lambda_1=(D_aD_bD_a)^2\\D_a^7D_b^{-1}D_aD_bD_a,\quad
		\Lambda_2=(D_aD_bD_a)^2D_a^3D_b^{-1}D_a^2D_bD_a.
	\]
An easy computation gives us}
\begin{align*}
	F_2(L(7,1)) =\counit^{F_2}\circ\rho^{F_2}_{\Lambda_1}\circ\unit^{F_2}
	& =145065/8\\
	& \neq 14295/2
		=\counit^{F_2}\circ\rho^{F_2}_{\Lambda_2}\circ\unit^{F_2}=F_2(L(7,2)).
\end{align*}
{ Similarly, we have}
\begin{align*}
		L(65,i)=\counit\circ\Cyl_{\Lambda_i}\circ \unit
	\quad\text{ for }
		\Lambda_8=\matrica{65 & 8\\8 & 1},\quad
		\Lambda_{18}=\matrica{65 & 18\\18 & 5}.
	\end{align*}
{ Therefore, we compute}
\begin{align*}
	F_2(L(65,8)) =\counit^{F_2}\circ\rho^{F_2}_{\Lambda_8}\circ\unit^{F_2}
	& =114462647835/4096\\
	& \neq 111887115/64
		=\counit^{F_2}\circ\rho^{F_2}_{\Lambda_{18}}\circ\unit^{F_2}=F_2(L(65,18)).
        \end{align*}
{ Likewise, we have that}
\begin{align*}
F_2(L(11,2))=264135/8&\neq88125/4=F_2(L(11,4)),\\
    F_2(L(13,3))=320835/8&\neq116475/4=F_2(L(13,5)),
\end{align*}
as claimed.
\end{proof}

\begin{rem}
	This restricted TQFT is unfortunately not faithful, as already the representation $\rho^{F_3}:\SL\to GL(3,\RR)$ factors through $PSL(2,\Z)$ by \cite{Tuba-Wenzl}.
\end{rem}

\appendix

\section{Appendix}
\label{sec:appendix}

{ The goal of this appendix is to prove Theorem~\ref{thm:freely} about the symmetric strict monoidal categories $\calC^\b$ and $\calC^\be$ introduced in Definition~\ref{sloboda}.} We consider these categories built out of syntax material, that is: the objects are of the form $\O^{\otm m}$ and the arrows are the equivalence classes of terms modulo equalities introduced in Definition~\ref{dfn:Tb-object} and \ref{defin:TbeObject} respectively, together with other equalities tied to symmetric strict monoidal structure of a category.

Let us abbreviate the identity morphism $\mj_\O$ by $\mj$ and $\mj_{\O^{\otm m}}$ by $\mj_m$. Let
\[
\tau_{m,n}:\O^{\otm m} \otm \O^{\otm n} \ra \O^{\otm n} \otm \O^{\otm m}
\]
for $m,n\geq 0$ be the symmetry (if $n=0$, then $\O^{\otm n}$ is the unit object). An arrow built out of $\tau$'s with the help of $\circ$ and $\otm$ is called a $\tau$-\emph{term}. Every $\tau$-term from $\O^{\otm n}$ to $\O^{\otm n}$ corresponds to a permutation in the symmetric group on $n$ letters. The strictness of $\calC^\b$ (resp.\ $\calC^\be$) implies
\begin{equation}\label{strict1}
\tau_{m,0}=\tau_{0,m}=\mj_m.
\end{equation}

Consider the following four types of arrows of $\calC^\b$ (with $A\in\SL$ varying):
\[
	C_1\coloneqq\beta^\O\circ (\rho^\O_A\otm\mj), \quad
	C_2\coloneqq\rho^\O_A, \quad
	C_3\coloneqq(\rho^\O_A\otm\mj)\circ \gamma^\O, \quad
	C_4\coloneqq\beta^\O\circ (\rho^\O_A\otm\mj)\circ \gamma^\O,
\]
and additional three types for $\calC^\be$ ({ for $A=\matrica{ p & r \\ q & s }\in\SL$, with $p\geq 0$:}):
\[
	C_0\coloneqq\counit^\O\circ\rho_A^\O\circ\unit^\O, \qquad C_5\coloneqq\rho_A^\O\circ\unit^\O, \qquad C_6\coloneqq\counit^\O\circ\rho_A^\O.
\]

\begin{dfn}
An arrow of $\calC^\b$ (resp.\ $\calC^\be$) is in \emph{normal form} when it is either $\mj_0$ or is obtained as the composition $\tau^t\circ C\circ \tau^s$, where $\tau^s$ and $\tau^t$ are $\tau$-terms and $C$ is a tensor product of arrows of the form $C_1-C_4$ (resp.\ $C_0-C_6$).
\end{dfn}

\begin{lem}\label{sulundar}
For every normal form $\nu=\tau^t\circ C\circ \tau^s: \O^{\otm n}\to \O^{\otm m}$ and every $j\in\{1,m-1\}$ there is a normal form $\tau^t_1\circ C_1\circ \tau^s_1$ equal to $\nu$ in which for some $i\in\{1,m-1\}$ the permutation corresponding to $\tau^t_1$ maps $i$ to $j$ and $i+1$ to $j+1$.
\end{lem}

\begin{proof}
Note that every factor of the tensor product $C$ has $\O^{\otm k}$ as the target, for $k\leq 2$. The symmetry is a natural isomorphism. This helps to reorder the factors in $C$ and if necessary one uses the equalities~\eqref{eq:tau-gamma} and~\eqref{eq:Tb-object-gamma} in order to obtain the desired form. The following example illustrates this procedure. Let $j=1$ and let $\nu$ be the following normal form:
\[
(\tau_{3,1}\otm \mj)\circ \Big(((\rho_A\otm\mj)\circ\gamma)\otm\rho_B\otm ((\rho_{D_b}\otm\mj)\circ\gamma)\Big).
\]
By the properties of symmetry, this can be easily reordered into the form:
\[
(\mj\otm \tau_{1,3})\circ \Big(((\rho_{D_b}\otm\mj)\circ\gamma)\otm ((\rho_A\otm\mj)\circ\gamma)\otm\rho_B\Big).
\]
Then we use~\eqref{eq:tau-gamma} and naturality of symmetry in order to obtain:
\[
\tau_{1,4}\circ \Big(((\mj\otm \rho_{D_b})\circ\gamma)\otm ((\rho_A\otm\mj)\circ\gamma)\otm\rho_B\Big).
\]
Finally, by using~\eqref{eq:Tb-object-gamma} we obtain the desired normal form, for $i=2$:
\[
\tau_{1,4}\circ \Big(((\rho_{D_a}\otm\mj)\circ\gamma)\otm ((\rho_A\otm\mj)\circ\gamma)\otm\rho_B\Big).\qedhere
\]
\end{proof}

\begin{prop}\label{nf}
    Every arrow of $\calC^\b$ (resp.\ $\calC^\be$) has a normal form.
\end{prop}

\begin{proof}
We give a proof for $\calC^\be$, which is more general. By relying on bifunctoriality of $\otm$, every arrow of this category is equal to a composition of arrows of the form $\mj_m\otm f\otm \mj_n$, where $f\in\{\mj_k, \tau, \rho_A^\O, \beta^\O,\gamma^\O, \unit^\O, \counit^\O\}$. Then an inductive argument boils down the proof to the case when the arrow is of the form $(\mj_m\otm f\otm \mj_n)\circ(\tau^t\circ C\circ \tau^s)$.

The cases when $f$ is $\mj_k$ or $\tau$ are trivial. If $f$ is $\rho_A^\O$, then we rely on naturality of symmetry, and if necessary on~\eqref{eq:Tb-object-gamma} { and~\eqref{eq:-E-e}}, in order to obtain a normal form. If $f$ is $\gamma^\O$ or $\unit^\O$, then (\ref{strict1}), together with naturality of symmetry, gives a normal form. If $f$ is $\counit^\O$, then the naturality of symmetry moves it through $\tau^t$ until it reaches $C$ and becomes confronted with a $C_2$ or $C_3$ or $C_5$ factor in it. { In the case of $C_5$ we are done, whereas for a $C_2$ factor, we might need to use~\eqref{eq:eta--E} to multiply the matrix by $-1$, making sure that its top left entry is non-negative.} In the case of a $C_3$ factor, we use Remark~\ref{rem:beta-e,eta-gamma} and, if necessary, we rely on~\eqref{eq:Tb-object-gamma} { and~\eqref{eq:-E-e}}.

If $f$ is $\beta^\O$, then we apply Lemma~\ref{sulundar}, symmetric monoidal coherence and naturality of symmetry in order to move this $f$ through $\tau^t$ until it reaches $C$ and becomes confronted with either (1) a pair of $C_2$ factors, or (2) one $C_2$ and one $C_3$ factor, or (3) one $C_2$ and one $C_5$ factor, or (4) one $C_3$ factor, or (5) a pair of $C_3$ factors, or (6) one $C_3$ and one $C_5$ factor, or (7) a pair of $C_5$ factors. We have ``ready to use'' equalities for all these cases, which is illustrated here just for the case (5). Let, for example, our $(\mj_m\otm f\otm \mj_n)\circ \tau^t\circ C\circ \tau^s$ be:
\[(\beta\otm\mj_3)\circ
(\tau_{3,1}\otm \mj)\circ \Big(((\rho_A\otm\mj)\circ\gamma)\otm\rho_B\otm ((\rho_{D_b}\otm\mj)\circ\gamma)\Big).
\]
By Lemma~\ref{sulundar} this is equal to:
\[
(\beta\otm\mj_3)\circ\tau_{1,4}\circ \Big(((\rho_{D_a}\otm\mj)\circ\gamma)\otm ((\rho_A\otm\mj)\circ\gamma)\otm\rho_B\Big).
\]
Symmetric monoidal coherence entails $\tau_{1,4}=(\mj_2\otm\tau_{1,2})\circ(\tau_{1,2}\otm \mj_2)$, which enables $\beta$ to move through this $\tau$-term by bifunctoriality of $\otm$ and naturality of $\tau$, giving:
\[
\tau_{1,2}\circ(\mj\otm \beta\otm\mj_2)\circ \Big(((\rho_{D_a}\otm\mj)\circ\gamma)\otm ((\rho_A\otm\mj)\circ\gamma)\otm\rho_B\Big).
\]
Then $\beta$ is coupled with one $\gamma$ in order to apply~\eqref{eq:snake}:
\[
\tau_{1,2}\circ \Big(\big((\rho_{D_a}\otm\mj)\circ ([(\mj\otm\beta) \circ (\gamma\otm\mj)]\otm\mj)\circ (\rho_A\otm\mj)\circ\gamma\big)\otm \rho_B\Big),
\]
which delivers the normal form:
\[
\tau_{1,2}\circ \Big(\big((\rho_{D_a\cdot A}\otm\mj)\circ\gamma\big)\otm \rho_B\Big).\qedhere
\]
\end{proof}

As a consequence of Lemma~\ref{lem:T-Tb} { and the fact that $\calC^\b$ and $\calC^\be$ are initial objects in the corresponding categories}, there are two functors $F^\beta:\calC^\b\to \bMCG$ and $F^\be:\calC^\be\to \beMCG$, which map $\O$ to the torus $T$ and respect the symmetric monoidal structure with $T^\beta$ (resp.\ $T^{\be}$)-object. We formulate the following lemma for $\calC^\be$, but one has an analogous result for $\calC^\b$.

\begin{lem}\label{1component}
If $f_1$ and $f_2$ are two arrows of $\calC^\be$ such that $F^\be(f_1)=F^\be(f_2)$ and this cobordism is connected, then $f_1=f_2$.
\end{lem}

\begin{proof}
By Proposition~\ref{nf}, we may assume that $f_1$ and $f_2$ are in normal form. Since the cobordism $F^\be(f_1)=F^\be(f_2)$ is connected we have that either $f_1=f_2=\mj_0$, or both $f_1$, $f_2$ contain exactly one factor among $C_0-C_6$. From the equality of cobordisms $F^\be(f_1)$ and $F^\be(f_2)$, one easily concludes that the unique factors in $f_1$, $f_2$ are of the same type. Moreover, by appealing to~\eqref{eq:tau-beta},~\eqref{eq:tau-gamma},~\eqref{eq:Tb-object} and~\eqref{eq:Tb-object-gamma}, we may assume that both $f_1$ and $f_2$ are free of $\tau$-terms, hence both are just arrows of type $i$, for $i\in\{0,\ldots,6\}$. Then the table

\begin{center}
\begin{tabular}{|l|c|}
\hline
The type of $f_1$ and $f_2$ & $f_1=f_2$ follows from: \\
\hline
	$C_0=\counit^\O\circ\rho_A^\O\circ\unit^\O$ & Theorem~\ref{thm:lens-spaces} and Proposition~\ref{prop:abstr-lens-space} \\
\hline
	$C_1=\beta^\O\circ (\rho^\O_A\otm\mj)$ 		& Corollary~\ref{beta-gamma-cyl} \\
\hline
	$C_2=\rho^\O_A$								& Corollary~\ref{cor:Cyl} \\
\hline
	$C_3=(\rho^\O_A\otm\mj)\circ \gamma^\O$ 	& Corollary~\ref{beta-gamma-cyl} \\
\hline
	$C_4=\beta^\O\circ (\rho^\O_A\otm\mj)\circ \gamma^\O$ & Corollary~\ref{cor:Bun-conj} and Proposition~\ref{prop:classify-new-proof} \\
\hline
	$C_5=\rho_A^\O\circ\unit^\O$	 & Lemma~\ref{cyl-epsilon} and~
	\eqref{eq:Da-e} \\
\hline
	$C_6=\counit^\O\circ\rho_A^\O$	 & Lemma~\ref{eta-cyl} and~\eqref{eq:eta-Db} \\
\hline
\end{tabular}
\end{center}
finishes the proof.
\end{proof}

\begin{thm}\label{thm:freely2}
	The functors $F^\beta:\calC^\b\to \bMCG$ and $F^\be:\calC^\be\to \beMCG$ are isomorphisms of categories.
\end{thm}

\begin{proof}
	We concentrate again on $F^\be$. It is evident that this functor is bijective on objects since it maps $\O^{\otm n}$ to $T^{\otm n}$. Moreover, it is clearly surjective on arrows. To show injectivity, let $f_1,f_2$ be two arrows such that $F^\be(f_1)=F^\be(f_2)$. If this cobordism is connected, then we apply Lemma~\ref{1component}. If it is not connected, then there is a bijective correspondence between its $m\geq 2$ components and $m$ factors in normal forms for $f_1$ and $f_2$.

	Since the factors corresponding to one component have the same ``ends'' in the sources and targets, by pre- and post-composing the normal forms with adequate $\tau$-terms (isomorphisms), we may assume that the normal form of $f_i$ is $f^1_i\otm\ldots\otm f^m_i$, where $f^j_1$ and $f^j_2$ correspond to the same component in $F^\be(f_1)=F^\be(f_2)$. By Lemma~\ref{1component}, we have that $f^j_1=f^j_2$, hence $f_1=f_2$.
\end{proof}

%%%%%%%%

%%%%%%%%


\begin{thebibliography}{99}

\bibitem{BCGP}
{\sc C.\ Blanchet, F.\ Costantino, N.\ Geer and B.\ Patureau-Mirand},
{\it Non-semi-simple {TQFTs}, {Reidemeister} torsion and {Kashaev}'s invariants},
\textbf{\textit{Advances in Mathematics}}, vol.\ {\bf 301}
(2016) pp.\ 1--78.
\url{http://doi.org/10.1016/j.aim.2016.06.003}


\bibitem{Constatino}
{\sc F.\ Costantino, N.\ Geer and B.\ Patureau-Mirand},
{Quantum invariants of 3-manifolds via link surgery presentations and non-semi-simple categories},
\textbf{\textit{Journal of Topology}} {\bf 7},
 no.~4 (2014) pp.\ 1005--1053.
\url{http://doi.org/10.1112/jtopol/jtu006}


\bibitem{E66}
{\sc D.~B.~A.\ Epstein},
{\it Curves on 2-manifolds and isotopies},
\textbf{\textit{Acta Mathematica}}, vol.\ {\bf 115}
(1966) pp.\ 83--107.
\url{http://doi.org/10.1007/BF02392203}

\bibitem{MCG}
{\sc B.\ Farb and D.\ Margalit},
\textbf{\textit{ A primer on mapping class groups}},
{Princeton Mathematical Series}, {\bf 49},
{Princeton University Press, Princeton, NJ},
2012

\bibitem{DGGPR}
{\sc M.\ De Renzi, A.~M.\ Gainutdinov, N.\ Geer, B.\ Patureau-Mirand and I.\ Runkel},
{\it Mapping class group representations from non-semisimple TQFTs},
\textbf{\textit{Communications in Contemporary Mathematics}} {\bf 25}
(2023) no.~1, ID 2150091, 52 pp.
\url{http://doi.org/10.1142/S0219199721500917}


\bibitem{DLN}
{\sc C.\ Dong, \sc X.\ Lin and S.-H.\ Ng},
{\it Congruence property in conformal field theory},
\textbf{\textit{Algebra Number Theory}} {\bf 9}
(2015) no.~9, pp.\ 2121--2166.
\url{http://doi.org/10.2140/ant.2015.9.2121}


\bibitem{FM97}
{\sc A.~T.\ Fomenko} and {\sc S.~V.\ Matveev},
\textbf{\textit{Algorithmic and Computer Methods for Three-Manifolds}},
Mathematics and Its Applications {\bf 425}, Springer Nature,
1997


\bibitem{Fosse}
{\sc D.\ Fosse},
{\it An explicit decomposition formula of a matrix in GL2(Z)},
\textbf{\textit{Journal of Algebra}} {\bf 637}
(2024) pp.\ 230-242.
\url{http://doi.org/10.1016/j.jalgebra.2023.09.015}


\bibitem{F13}
{\sc L.\ Funar},
{\it Torus bundles not distinguished by TQFT invariants},
\textbf{\textit{Geometry \& Topology}}, vol.\ {\bf 17}
(2013) pp.\ 2289-2344.
\url{http://doi.org/10.2140/gt.2013.17.2289}


\bibitem{F23}
{\sc L.\ Funar},
{\it On mapping class groups and their TQFT representations},
2023 preprint available at \url{https://arxiv.org/abs/2302.02883}


\bibitem{G20}
{\sc S.\ Gajovi\'c, Z.\ Petri\'c{} and S.\ Telebakovi\'c~Oni\'c},
{\it A faithful 2-dimensional TQFT},
\textbf{\textit{Homology, Homotopy and Applications}} {\bf 22}
(2020) no.~1, pp.\ 391--399.
\url{https://doi.org/10.4310/HHA.2020.v22.n1.a22}


\bibitem{Hatcher}
{\sc A.\ Hatcher},
{\it Notes on Basic 3-Manifold Topology},
available at \url{https://pi.math.cornell.edu/~hatcher/3M/3Mdownloads.html}


\bibitem{Jeffrey}
{\sc L.~C.\ Jeffrey},
{\it Chern-Simons-Witten invariants of lens spaces and torus bundles, and the semiclassical approximation},
\textbf{\textit{Communications in Mathematical Physics}},
vol.\ {\bf 147} (1992) pp.\ 563-604.
\url{https://doi.org/10.1007/BF02097243}


\bibitem{Juhasz}
{\sc A.\ Juh\'asz},
{\it Defining and classifying TQFTs via surgery},
\textbf{\textit{Quantum Topology}}
{\bf 9} (2018) no.~2, pp.\ 229--321.
\url{https://doi.org/10.4171/QT/108}


\bibitem{Karpenkov}
{\sc O.~N.\ Karpenkov},
\textbf{\textit{Geometry of continued fractions}},
second edition,  Algorithms and Computation in Mathematics, {\bf 26}, Springer, Berlin,
2022


\bibitem{ML63}
{\sc S.\ Mac Lane},
{\it Natural associativity and commutativity},
\textbf{\textit{Rice University Studies, Papers in Mathematics}},
vol.\ {\bf 49} (1963) pp.\ 28--46.


\bibitem{ML98}
{\sc S.\ Mac Lane},
\textbf{\textit{Categories for the Working Mathematician}},
second edition, Springer, Berlin,
1998


\bibitem{McC}
{\sc D.\ McCullough},
{\it Homeomorphisms which are Dehn twists on the boundary},
\textbf{\textit{Algebraic \& Geometric Topology}},
vol.\ {\bf 6.3} (2006) pp.\ 1331-1340.
\url{https://doi.org/10.2140/agt.2006.6.1331}


\bibitem{M65}
{\sc J.\ Milnor},
\textbf{\textit{Lectures on the h-cobordism theorem}},
Princeton University Press,
1965


\bibitem{Muller-Woike}
{\sc L.\ Müller and L.\ Woike},
{\it The Dehn Twist Action for Quantum Representations of Mapping Class Groups},
to appear in \textbf{\textit{Journal of Topology}},
2023 preprint available at \url{https://arxiv.org/abs/2311.16020}


\bibitem{N81}
{\sc W.~D.\ Neumann},
{\it A calculus for plumbing applied to the topology of complex surface singularities and degenerating complex curves},
\textbf{\textit{Transactions of the American Mathematical Society}},
vol.\ {\bf 268} (1981) pp.\ 299-344.
\url{https://doi.org/10.2307/1999331}


\bibitem{N05}
{\sc W.~D.\ Neumann},
{\it Graph 3-manifolds, splice diagrams, singularities},
\textbf{\textit{Singularity theory}},
(2007) pp.\ 787--817,
World Sci. Publ., Hackensack, NJ


\bibitem{NPZ}
{\sc J.\ Nikoli\' c, Z.\ Petri\' c and M.\ Zeki\' c},
{\it A diagrammatic presentation of the category 3Cob},
\textbf{\textit{Results in Mathematics}},
vol.\ {\bf 79} (2024) no.~4, 29 pp.
\url{https://doi.org/10.1007/s00025-024-02201-8}


\bibitem{PS}
{\sc V.~V.\ Prasolov and A.~B.\ Sosinski\u \i},
\textbf{\textit{Knots, Links, Braids and 3-Manifolds}},
Translations of Mathematical Monographs 154, American Mathematical Society,
1996


\bibitem{Resh-Turaev}
{\sc N.~Y.\ Reshetikhin and V.~G.\ Turaev},
{\it Invariants of $3$-manifolds via link polynomials and quantum groups},
\textbf{\textit{Inventiones Mathematicae}}
{\bf 103} (1991) no.~3, pp. 547--597.
\url{https://doi.org/10.1007/BF01239527}


\bibitem{Sokolov}
{\sc M.~V.\ Sokolov},
{\it Which lens spaces are distinguished by Turaev-Viro invariants}, (in Russian)
\textbf{\textit{Mat. Zametki}},
vol.\ {\bf 61} (1997) pp.\ 468--470; Translation in \textbf{\textit{Math. Notes}}, vol.\ {\bf 61} (1997) pp.\ 384--387.


\bibitem{Stebe}
{\sc P.~F.\ Stebe},
{\it Conjugacy separability of groups of integer matrices},
\textbf{\textit{Proceedings of the American Mathematical Society}},
vol.\ {\bf 32} (1972) pp.\ 1--7.
\url{https://doi.org/10.2307/2038292}


\bibitem{Tuba-Wenzl}
{\sc I.\ Tuba and H.\ Wenzl},
{\it Representations of the Braid Group $B_3$ and of $SL(2,Z)$.}
\textbf{\textit{Pacific Journal of Mathematics}},
vol.\ {\bf 197}, (2001) no. 2, pp.\ 491-–510.
\url{https://doi.org/10.2140/pjm.2001.197.491}


\bibitem{Turaev}
{\sc V.~G.\ Turaev},
\textbf{\textit{Quantum invariants of knots and 3-manifolds}},
De Gruyter Studies in Mathematics, {\bf 18}, de Gruyter, Berlin,
1994


\bibitem{W67a}
{\sc F.\ Waldhausen},
{\it Gruppen mit Zentrum und 3-dimensionale Mannigfaltigkeiten},
\textbf{\textit{Topology}},
vol.\ {\bf 6} (1967) pp.\ 505–-517.
\url{https://doi.org/10.1016/0040-9383(67)90008-0}


\bibitem{W67b}
{\sc F.\ Waldhausen},
{\it Eine Klasse von 3-dimensionalen Mannigfaltigkeiten. I, II.},
\textbf{\textit{Inventiones Mathematicae}},
vol.\ {\bf 3} (1967) pp.\ 308--333;
vol.\ {\bf 4} (1967) pp.\ 88--117.
\url{https://doi.org/10.1007/BF01402956}

\bibitem{Walker}
{\sc K.\ Walker},
{\it On Witten's 3-Manifold Invariants},
\url{http://canyon23.net/math}, 1991

\end{thebibliography}
\end{document}